\documentclass[12pt]{amsart}

\usepackage{amsfonts, amsthm, amsmath, amssymb}

\usepackage{makecell,multirow,diagbox}

\usepackage{amscd}

\usepackage[latin2]{inputenc}

\usepackage{t1enc}

\usepackage[mathscr]{eucal}

\usepackage{indentfirst}

\usepackage{graphicx}

\usepackage{graphics}

\usepackage{pict2e}

\usepackage{mathrsfs}

\usepackage{framed}

\usepackage{mathabx}
\usepackage{booktabs}
\usepackage{makecell}
\usepackage{enumerate}
\usepackage[pagebackref]{hyperref}
\hypersetup{colorlinks=true}
\usepackage{cite}
\usepackage{color}
\usepackage{epic}
\usepackage{hyperref} 
\numberwithin{equation}{section}
\def\red{\textcolor{red}}

\theoremstyle{plain}

\usepackage{tikz}

\usetikzlibrary{calc}
\usetikzlibrary{shapes.geometric}

\tikzset{
  c/.style={every coordinate/.try}
}

\usetikzlibrary{arrows,shapes,positioning}
\usetikzlibrary{decorations.markings}
\tikzstyle arrowstyle=[scale=1]
\tikzstyle directed=[postaction={decorate,decoration={markings,mark=at position 0.6 with {\arrow[arrowstyle]{stealth};}}}]
\tikzstyle reverse directed=[postaction={decorate,decoration={markings,mark=at position 0.4 with {\arrowreversed[arrowstyle]{stealth};}}}]
\tikzstyle dot=[style={circle,inner sep=1pt,fill}]

\newtheorem{theorem}{Theorem}[section]

\newtheorem{lemma}[theorem]{Lemma}

\newtheorem{corollary}[theorem]{Corollary}

\newtheorem{St}[theorem]{Statement}

\newtheorem{proposition}[theorem]{Proposition}

\theoremstyle{definition}

\newtheorem{example}[theorem]{Example}

\newtheorem{conj}[theorem]{Conjecture}

\newtheorem{?}[theorem]{Problem}

\def\St{\mathrm{St}}
\def\st{\mathrm{st}}
\def\SS{\mathfrak{S}}

\def\asc{\mathrm{asc}}
\def\Asc{\mathrm{Asc}}
\def\Dt{\mathrm{Dt}}
\def\Des{\mathrm{Des}}
\def\des{\mathrm{des}}
\def\max{\mathrm{max}}
\def\min{\mathrm{min}}

\def\Pk{\mathrm{Pk}}
\def\Va{\mathrm{Va}}
\def\Sf{\mathrm{Sf}}
\def\Su{\mathrm{Su}}
\def\Fix{\mathrm{Fix}}
\def\Tl{\mathrm{Tl}}
\def\Tr{\mathrm{Tr}}

\def\tl{\mathrm{tl}}
\def\tr{\mathrm{tr}}
\def\pk{\mathrm{pk}}
\def\va{\mathrm{va}}
\def\sf{\mathrm{sf}}
\def\su{\mathrm{su}}

\def\I{{\bf I}}
\def\N{\mathbb{N}}

\newcommand{\tabincell}[2]{
\begin{tabular}{@{}#1@{}}#2\end{tabular}
}

\def\boxit#1{\leavevmode\hbox{\vrule\vtop{\vbox{\kern.33333pt\hrule
    \kern1pt\hbox{\kern1pt\vbox{#1}\kern1pt}}\kern1pt\hrule}\vrule}}

\usepackage{collectbox}



\begin{document}

\title[Symmetries of ascents in inversion sequences]{Combinatorics of the symmetries of ascents in restricted  inversion sequences}

\author[J.N. Chen]{Joanna N. Chen}
\address[Joanna N. Chen]{College of Science, Tianjin University of Technology, Tianjin 300384, P.R. China}
\email{joannachen@tjut.edu.cn}

\author[Z. Lin]{Zhicong Lin}
\address[Zhicong Lin]{Research Center for Mathematics and Interdisciplinary Sciences, Shandong University, Qingdao 266237, P.R. China}
\email{linz@sdu.edu.cn}

\date{\today}
\begin{abstract}
The systematic study of inversion sequences avoiding triples of relations was initiated by Martinez and Savage. For a triple  $(\rho_1,\rho_2,\rho_3)\in\{<,>,\leq,\geq,=,\neq,-\}^3$, they introduced  $\I_n(\rho_1,\rho_2,\rho_3)$ as the set  of inversion sequences $e=e_1e_2\cdots e_n$ of length $n$ such that there are no indices $1\leq i<j<k\leq n$ with $e_i \rho_1 e_j$, $e_j \rho_2 e_k$ and $e_i \rho_3 e_k$. To solve a conjecture of Martinez and Savage, Lin constructed a bijection between $\I_n(\geq,\neq,>)$ and $\I_n(>,\neq,\geq)$ that preserves the distinct entries and further posed a symmetry conjecture of ascents on these two classes of restricted inversion sequences. Concerning Lin's symmetry conjecture, an algebraic proof  using the  kernel method was recently provided by Andrews and Chern, but a bijective proof still remains mysterious. The goal of this article is to establish bijectively both Lin's symmetry conjecture and the $\gamma$-positivity of the ascent polynomial on $\I_n(>,\neq,>)$. The latter result implies that the distribution of ascents on $\I_n(>,\neq,>)$ is symmetric  and unimodal.
\end{abstract}
\keywords{Inversion sequences; Pattern avoidance; Bijective proof; $b$-code; $\gamma$-positive}
\maketitle
%

\section{Introduction}\label{sec:intro}

An \emph{inversion sequence} is a sequence $e=e_1 e_2 \cdots e_n$ of natural numbers such that
$0 \leq e_i \leq i-1$ for all $i \in [n]:=\{1,2,\cdots, n\}$.
Let $\I_n$ be the set of all inversion sequences of length $n$.
Denote by $\SS_n$  the set of permutations of $[n]$. There are several bijections from $\SS_n$ to $\I_n$, which can be seen as natural
codings of permutations. For example,  the map $\Theta: \SS_n \rightarrow \I_n $ defined for $\pi=\pi_1 \pi_2 \cdots \pi_n \in \SS_n$ as
\[\Theta(\pi)=e_1 e_2 \cdots e_n\quad\text{with}\quad  \,\,\,\, e_i=|\{j : j<i  \text{ and } \pi_j >\pi_i\}|,\]
 is a bijection known as the \emph{Lehmer code} of permutations. Many intriguing interactions between inversion sequences and permutations have been exploited in the literature~\cite{auli2,fjlyz,HZ,kl,kl2,ly,SV}.

Recently, inversion sequences with various restrictions  have been extensively studied from their purely enumerative aspect.
Inversion sequences avoiding a single pattern of length $3$ were first considered  by
Corteel, Martinez, Savage and Weselcouch~\cite{CMSW} and Mansour and Shattuck \cite{Mansour}, independently.
Martinez and  Savage~\cite{MS} further generalized the classical patterns of length $3$ to patterns of relation triples.
Their works have inspired lots of further investigations, including refinements by statistics in~\cite{Andrews,bbgr,cjl,kl,kl2,yan}, consecutive patterns by Auli and Elizalde~\cite{auli,auli2}, vincular patterns in~\cite{auli3,ly,fl}, and pairs of length-$3$ patterns by Yan and Lin~\cite{ylin}. The objective of this article is to prove combinatorially symmetries of the ascent statistic for several classes of  inversion sequences avoiding patterns of relation triples.

Permutations and inversion sequences are viewed as words over $\N:=\{0,1,2,\ldots\}$.
A word $w=w_1\cdots w_n$ is said to \emph{avoid} a pattern $p=p_1\cdots p_k$ (or {\em$p$-avoiding}) if none of the subsequences of $w$ are order isomorphic to $p$. Otherwise, $w$ is said to \emph{contain} a pattern $p$ (or a $p$-pattern).
For a set $\mathcal{W}$ of words, denote by  $\mathcal{W}(p_1,p_2,\ldots,p_r)$  the set of words in $\mathcal{W}$ avoiding patterns $p_1, p_2, \ldots, p_r$.
For inversion sequences, Martinez and  Savage~\cite{MS} generalized  the notation of a pattern of length $3$ to a fixed triple of binary relations $(\rho_1,\rho_2,\rho_3)$. For each relation triple $(\rho_1,\rho_2,\rho_3) \in \{<,>,\leq, \geq, =, \neq, -\}^3$, they considered the set $\I_n(\rho_1,\rho_2,\rho_3)$ consisting of those $e\in\I_n$ with no $i<j<k$ such that $e_i \rho_1 e_j$, $e_j \rho_2 e_k$ and $e_i \rho_3 e_k$. For example, $\I_n(<,>,<)=\I_n(021)$ and
$\I_n(\geq,\neq,>)=\I_n(110,201,210)$. In general, patterns of relation triples are some special multiple patterns of length $3$.


For a word $w=w_1 w_2 \cdots w_n$ over $\N$, an index $i\in[n-1]$ is called an \emph{ascent} of $w$ if $w_i < w_{i+1}$, a \emph{descent} if $w_i>w_{i+1}$ and a {\em plateau} if $w_i=w_{i+1}$. The number of ascents over inversion sequences is an {\em Eulerian statistic} (see~\cite{kl}), which is one of the most important  statistics studied in the literature. Throughout the paper, we make the convention that all sets can be viewed as multisets and elements in a set are ordered nondecreasingly.  If $\St$ is a set-valued statistic, then $\st$ is defined to be the corresponding numerical statistic. If $S$ is a set, then $x^S$ stands for the monomial $\prod_{i\in S}x_i$.
 Let us introduce three useful set-valued statistics
 \begin{align*}
  & \Dt(w)  =\{w_i \colon 1 \leq i <n \text{ and } w_i >w_{i+1}\}, \\[3pt]
  & \Des(w)  =\{i \colon 1 \leq i <n \text{ and } w_i >w_{i+1}\}, \\[3pt]
  & \Asc(w) =\{i \colon 1 \leq i <n \text{ and } w_i < w_{i+1}\},
 \end{align*}
called the multiset of {\em descent tops},  the set of {\em descent positions} and the set of {\em ascent positions} of $w$, respectively.

The integer sequence A098746
$$
\{1, 1, 2, 6, 23, 102, 495, 2549, 13682, 75714,428882,\ldots\}
$$
in the OEIS~\cite{oeis} defined by the algebraic generating function
$$
(t^2-t+1)A(t)^3+(t-3)A(t)^2+3A(t)-1=0
$$
has plenty of combinatorial interpretations (see the talk~\cite{bur} by Burstein at PP2020) in terms of
\begin{itemize}
\item forests of planted ternary trees;
\item permutations avoiding one of $16$ $(4,5)$-symmetry classes of pairs of patterns, one of which is the pair $(4231,42513)$ closely related to the strictly locked jump queue~\cite{alb};
\item and three classes of inversion sequences avoiding relation triples found by Martinez and Savage~\cite{SV}, namely
$$
\I_n(>,-,>),\quad \I_n(\geq,\neq,>)\quad\text{and}\quad\I_n(>,\neq,\geq).
$$
\end{itemize}
In fact, Martinez and Savage~\cite{SV} proved bijectively that
\begin{equation}\label{ms:equi}
|\I_n(>,-,>)|=|\I_n(\geq,\neq,>)|
\end{equation}
and further conjectured that
\begin{equation}\label{ms:conj}
|\SS_n(4231,42513)|=|\I_n(>,-,>)|=|\I_n(\geq,\neq,>)|=|\I_n(>,\neq,\geq)|.
\end{equation}
This conjecture was confirmed by Lin in~\cite{Lin}: the first identity was established algebraically  by decomposing  $(>,-,>)$-avoiding inversion sequences, while the third one was proved via constructing a bijection  between $\I_n(\geq,\neq,>)$ and $\I_n(>,\neq,\geq)$. Lin's bijection does preserve some natural statistics but could not prove his symmetry conjecture of ascent statistic below.

\begin{conj}[Lin's symmetry conjecture] \label{conj:Lin}
For $n \geq 1$,
\begin{equation*}
 \sum_{e \in \I_n(\geq, \neq, >)} t ^{\asc(e)}= \sum_{e \in \I_n(>, \neq, \geq)} t ^{n-1-\asc(e)}.
\end{equation*}
\end{conj}

Very recently, Andrews and Chern~\cite{Andrews} confirmed the above conjecture by generating functions. They decomposed inversion sequences by considering the left-most appearance of the largest entry, translating into functional equations and solving the resulting equations using kernel method. The two kernels are polynomials with  order $4$, which makes the discussions unexpectedly complicated. They remarked that a bijective proof of Lin's symmetry conjecture still remains mysterious.

In this paper, we will prove Conjecture~\ref{conj:Lin} by bijectively proving the following two Theorems.

\begin{theorem}\label{thm:110100}
For $n \geq 1$,
\begin{equation*}
 \sum_{e \in \I_n(\geq, \neq, >)} t ^{\Asc(e)}= \sum_{e \in \I_n(>, -, >)} t^{\Asc(e)}.
\end{equation*}
\end{theorem}

\begin{theorem}\label{thm:100101}
For $n \geq 1$,
\begin{equation*}
\sum_{e \in \I_n(>, -, >)} t^{\asc(e)} q^{\Dt(e)}= \sum_{e \in \I_n(>, \neq, \geq)} t ^{n-1-\asc(e)} q^{\Dt(e)}.
\end{equation*}
\end{theorem}

Meanwhile, we have the following  corollary as a by-product.


\begin{corollary}\label{cor:intersect}
For $n \geq 1$,
\begin{equation*}
\sum_{e \in \I_n(>, -, \geq)} t^{\asc(e)} q^{\Dt(e)}= \sum_{e \in  \I_n(>, -, \geq)} t ^{n-1-\asc(e)} q^{\Dt(e)}.
\end{equation*}
\end{corollary}

A polynomial $h(t)=\sum_{i=0}^dh_it^i$ is said to be {\em unimodal} if the coefficients are increasing and then decreasing, i.e., there is an index $c$ such that $h_0\leq h_1\leq\cdots\leq h_c\geq h_{c+1}\geq\cdots\geq h_d$.  It is called {\em symmetric} (or {\em palindromic}) if $h_i=h_{n-i}$ for $0\leq i\leq d/2$. A nice property that implies symmetry and unimodality of a polynomial $h(t)$ is the so-called {\em$\gamma$-positivity}~\cite{Ath}, i.e., admits an expansion in the basis $\{t^k(1+t)^{d-k}\}_{0\leq k\leq d/2}$ with non-negative coefficients. Corollary~\ref{cor:intersect} implies the symmetry of  $\sum_{e \in \I_n(>, -, \geq)} t^{\asc(e)}$, a polynomial which turns out to be $\gamma$-positive from the previous works~\cite{kl,kl2} by Kim and Lin.

An index $i\in[n-2]$ is called a {\em double ascent} of an inversion sequence $e\in\I_n$ if $\{i,i+1\}\subseteq\Asc(e)$. Let
$$
\widetilde{\I_{n,k}}:=\{e\in\I_n: \asc(e)=k, \text{ $e$ has no double ascents and $e_{n-1}\geq e_{n}$}\}
$$
and let $\widetilde{\I_{n,k}}(\rho_1,\rho_2,\rho_3)=\I_n(\rho_1,\rho_2,\rho_3)\cap\widetilde{\I_{n,k}}$ for any relation triple $(\rho_1,\rho_2,\rho_3)$.
\begin{proposition}[\text{Kim and Lin~\cite[Thms~3.6 and~3.7]{kl2} and~\cite[Eq.~(4.10)]{kl}}]
\label{pro:kl}
For $n\geq1$,
\begin{equation}\label{gam:kl}
\sum_{e \in \I_n(>, -, \geq)} t^{\asc(e)}=\sum_{k=0}^{\lfloor(n-1)/2\rfloor}|\widetilde{\I_{n,k}}(>, -, \geq)|t^k(1+t)^{n-1-2k}.
\end{equation}
\end{proposition}
As an application of the Lehmer code and the {\em Foata--Strehl actions}~\cite{fsh} on permutations, we give a neat proof of~\eqref{gam:kl}. A similar approach, with the Lehmer code replaced by the $b$-code due to Baril and Vajnovszki~\cite{Baril},  is available to prove the following $\gamma$-positivity of $\sum_{e \in \I_n(>,\neq,>)} t^{\asc(e)}$.

\begin{theorem}
\label{thm:gamcl}
For $n\geq1$,
\begin{equation}\label{gam:cl}
\sum_{e \in \I_n(>,\neq,>)} t^{\asc(e)}=\sum_{k=0}^{\lfloor(n-1)/2\rfloor}|\widetilde{\I_{n,k}}(>,\neq,>)|t^k(1+t)^{n-1-2k}.
\end{equation}
In particular,
$$
\sum_{e \in \I_n(>,\neq,>)} t^{\asc(e)}=\sum_{e \in \I_n(>,\neq,>)} t^{n-1-\asc(e)}.
$$
\end{theorem}

It is convenient to point out the relationships among all the classes of restricted inversion sequences appearing in this paper. Setting $\mathcal{A}_n=\I_n(\geq,\neq,>)$, $\mathcal{B}_n=\I_n(>,\neq,\geq)$ and $\mathcal{C}_n=\I_n(>,-,>)$, then
\begin{align*}
&\mathcal{A}_n=\I_n(201,210,110), \quad\mathcal{B}_n=\I_n(201,210,101),\quad \mathcal{C}_n=\I_n(201,210,100), \\
&\mathcal{A}_n,\mathcal{B}_n,\mathcal{C}_n\subseteq \I_n(>,\neq,>)=\I_n(201,210),\\
&\mathcal{A}_n\cap\mathcal{B}_n=\I_n(201,210,110,101)=\I_n(\geq,\neq,\geq),\\
&\mathcal{B}_n\cap\mathcal{C}_n=\I_n(201,210,100,101)=\I_n(>,-,\geq),\\
&\mathcal{C}_n\cap\mathcal{A}_n=\I_n(201,210,110,100)=\I_n(\geq,-,>),\\
&\mathcal{A}_n\cap\mathcal{B}_n\cap\mathcal{C}_n=\I_n(201,210,110,101,100).
\end{align*}
Note that the three classes  $\I_n(\geq,\neq,\geq)$, $\I_n(>,-,\geq)$ and $\I_n(\geq,-,>)$ were known~\cite{kl2,MS} to be enumerated by the {\em large Schr\"oder numbers} (see~\cite[A006318]{oeis}). The ascent polynomial on $\I_n(\geq,\neq,\geq)$ was known to be $\gamma$-positive~\cite{kl,kl2}, while the ascent polynomial on $\I_n(\geq,-,>)$ is not symmetric. The class $\I_n(201,210,110,101,100)$ turns out (proved in Proposition~\ref{pro:fine}) to be counted by the {\em binomial transformation of the Fine's sequence} (see~\cite[A033321]{oeis}) and the ascent polynomial over this class is not symmetric.

The rest of this paper is organized as follows. In Section~\ref{sec:bijection}, we give the bijective proofs of Theorems~\ref{thm:110100} and~\ref{thm:100101}. As applications of the Lehmer code, $b$-code and the Foata--Strehl action for permutations, we prove combinatorially Proposition~\ref{pro:kl} and Theorem~\ref{thm:gamcl} in Section~\ref{gam:posi}.

\section{A bijective proof of Lin's symmetry conjecture}\label{sec:bijection}

In this section, we will give  bijective proofs of Theorem \ref{thm:110100} and Theorem \ref{thm:100101}, which in turn
settle bijectively  Conjecture  \ref{conj:Lin}.

\subsection{Bijective proof of Theorem \ref{thm:110100}}
To prove Theorem \ref{thm:110100}, we recall a bijection $\alpha$ from  $\I_n(\geq, >, -)$  to $\I_n(>, \geq, -)$ and its inverse $\beta$
 from  $\I_n(>, \geq,  -)$  to $\I_n(\geq, >, -)$ given by Martinez and  Savage \cite{MS}.
Notice that $\I_n(\geq, >, -)=\I_n(110,210)$ and $\I_n(>, \geq, -)=\I_n(100,210)$. For $e \in \I_n(110,210)$, let $\alpha(e)=t$ with
\begin{equation*}
t_j= \left\{
  \begin{array}{ll}
\max\{e_1,  \ldots,e_j\},  &\mbox{\text{if  $e_j=e_k$ for some $k>j$;}} \\[6pt]
e_j, & \mbox{\text{otherwise.}}
 \end{array} \right.
\end{equation*}
For $t \in \I_n(100,210)$,
let $\beta(t)=e$ with
\begin{equation*}
e_j= \left\{
  \begin{array}{ll}
\min\{t_j,  \ldots,t_n\},  &\mbox{\text{if  $t_i=t_j$ for some $i<j$;}} \\[6pt]
t_j, & \mbox{\text{otherwise.}}
 \end{array} \right.
\end{equation*}

\begin{lemma}\label{lem:keepasc}
For $n \geq 1$ and $e\in \I_n(110,210)$, we have $\Asc(e)=\Asc(\alpha(e))$.
\end{lemma}

\begin{proof}
Assume that $e\in\I_n(110,210)$ and $t=\alpha(e) \in \I_n(100,210)$.
It suffices  to show that
$j \in \Asc(e)$ if and only if $j \in \Asc(t)$.

If $j \in \Asc(e)$,  we claim that $t_j<e_{j+1}$. We consider two cases.
If there is no $k>j$ such that $e_j=e_k$, then clearly  we have $t_j=e_j<e_{j+1}$. If $e_j=e_k$ for some $k>j$,
 then we have $e_i<e_{j+1}$ for $i<j$. Otherwise, $e_i e_{j+1} e_k$ is an instance of patterns $210$ or $110$ . Hence, $t_j=
\max\{e_1, \ldots, e_j\}<e_{j+1}$. The claim is verified.
 Notice that $e_{j+1} \leq t_{j+1} $. It follows that
 $t_j<t_{j+1}$, namely, $j \in \Asc(t)$.

 If $j \in \Asc(t)$,  we claim that $e_{j+1} > t_{j}$. There are  two cases to consider.
If there is no $k<j+1$ such that $t_k=t_{j+1}$, then  we have $e_{j+1}=t_{j+1}>t_j$. If $t_k=t_{j+1}$ for some $k<j+1$,
 then we deduce that $t_i>t_{j}$ for $i>j+1$. Otherwise, $t_k t_{j} t_i$ forms pattern $210$ or $100$. Hence, $e_{j+1}=
\min\{t_{j+1}, \ldots, t_n\}>t_{j}$. The claim is verified.
 Since $e_{j } \leq t_{j} $, we have
 $e_{j+1}> e_j$, namely, $j \in \Asc(e)$. The proof is now completed.
\end{proof}

It was shown by Martinez and  Savage \cite{MS} that
\begin{align*}
  \I_n(\geq, \neq, >) &=\I_n(110,210,201)  \subseteq \I_n(110,210)  =\I_n(\geq ,>,-) \\[3pt]
  \I_n(>, -, >) &=\I_n(100,210,201)  \subseteq \I_n(100,210)  =\I_n(>, \geq ,-)
\end{align*}
and $\alpha$ can be restricted to a bijection from
$\I_n(\geq, \neq, >)$ to $\I_n(>,-,>)$. Hence, Theorem \ref{thm:110100} follows directly from Lemma \ref{lem:keepasc}.

\subsection{ Bijective proof of Theorem~\ref{thm:100101}}
In this section, we  construct  a bijection $$\gamma:\I_n(>,-,>)=\I_n(100,210,201)\rightarrow\I_n(>,\neq, \geq)=\I_n(101,210,201)$$
satisfying
$$
\asc(e)=n-1-\asc(\gamma(e))\,\,\text{ and }\,\,\Dt(e)=\Dt(\gamma(e))\quad\text{for each $e\in\I_n(>,-,>)$},
$$
 which implies Theorem~\ref{thm:100101}.
 Firstly, we recall a bijection $\Psi$ from $\I_n(100,210,201)$ to $\I_n(101,210,201)$ and its inverse $\Psi^{-1}$ given by Burstein~\cite{bur}.

Given $e\in \I_n(100,210,201)$, $\Psi(e)$ can be obtained as follows. For each $i=1,2,\ldots, n$, in that order,
\begin{itemize}
  \item check if $e_i$ is the second $1$ in an instance of pattern $101$ in $e$,
  \item if it is, change $e_i$ so as to turn this instance of pattern $101$ into an instance
of pattern $100$ on the same letters. Otherwise, leave $e_i$ unchanged.
\end{itemize}
For example, if $e=01021332343\in\I_{11}(100,210,201)$, then
$$e\rightarrow0102\red{0}332343\rightarrow0102033\red{0}343\rightarrow01020330\red{0}43\rightarrow0102033004\red{0}=\Psi(e).$$
Conversely, for $e\in \I_n(101,210,201)$, $\Psi^{-1}(e)$ can be obtained as follows. For each $i=n,\ldots, 2,1$, in that order,
\begin{itemize}
  \item check if $e_i$ is the second 0 in an instance of pattern $100$ in e,
  \item if it is, find the maximal $1$ among such instances of $100$ and change $e_i$ so
      as to turn that instance of pattern $100$ into an instance of pattern $101$ on the same letters. Otherwise, leave $e_i$ unchanged.
\end{itemize}

For $e\in \I_n$, let $e_0=e_{n+1}=+\infty$ throughout this section. Define the set of {\em peaks} and  {\em valleys} of $e$ by
\begin{align*}
  \Pk(e) & =\{i\in[n]: e_{i-1}< e_i \geq e_{i+1}\}, \\[3pt]
  \Va(e) & =\{i\in[n]: e_{i-1} \geq e_i <e_{i+1}\}.
\end{align*}
The proposition below can be easily checked.
\begin{proposition}\label{prop:gamma}
For $e \in \I_n$,  we have
$\va(e)-\pk(e)  = 1$.
\end{proposition}

 Given  $e\in \I_n(100,210,201)$, define the set of {\em special fixed positions} and {\em special unfixed positions} as
\begin{align*}
   \Sf(e) & =\{i\in[n]: e_{i-1}<e_{i-2}=e_i  \neq e_{i+1}\},\\[3pt]
    \Su(e) & =\{i\in[n]:  e_{i-1}<e_{i-2}=e_i  = e_{i+1}\}.
\end{align*}
Then the set of {\em fixed positions},  {\em to-right positions} and {\em to-left positions} of $e$ are defined by
\begin{align*}
  \Fix(e) & = (\Pk(e) \setminus \Su(e))\cup \Va(e) \cup \Sf(e), \\[3pt]
  \Tr(e)  &=\{ i \in [n] \setminus \Fix(e): e_{i-1}=e_i \} \cup \{1 \colon e_1=e_2=0\} \cup \Su(e),\\[3pt]
  \Tl(e) & =\{ i \in [n] \setminus \Fix(e): e_{i-1}<e_i \}.
\end{align*}
These three set-valued statistics will play important roles in our construction of $\gamma$.
A letter $e_i$ of $e$ is called a \emph{fixed }(resp.~\emph{to-right}, \emph{to-left}) {\em element} if $i \in \Fix(e)$ (resp.~$i \in \Tr(e)$, $i \in \Tl(e)$); it is called \emph{crucial} if
$e_{i}=e_{i-2}>e_{i-1}$, i.e., $e_{i-2}e_{i-1}e_i$ forms a $101$-pattern.

\begin{example}\label{exam:1}
For $e=000033033346\in\I_{12}(100,210,201)$, we compute that
\begin{equation*}
\Pk(e)=\{5,8\}, \quad\Va(e)=\{4,7,10\},\quad\Sf(e)=\emptyset\quad\text{and}\quad\Su(e)=\{8\}.
\end{equation*}
Thus,
$$
\Fix(e)=\{4,5,7,10\},\quad\Tr(e)=\{1,2,3,6,8,9\}\quad\text{and}\quad\Tl(e)=\{11,12\}.
$$
\end{example}

Clearly, the three sets $\Fix(e)$, $\Tr(e)$ and $\Tl(e)$ are disjoint. In fact, they form a partition of $[n]$ as proved below.
\begin{proposition}\label{prop:gamma1}
For $e \in \I_n(100,210,201)$,  we have $\Fix(e) \cup\Tr(e) \cup \Tl(e)=[n]$.

%
\end{proposition}

\begin{proof}
Notice that
$\Fix(e) \cup \Tr(e) \cup \Tl(e) \subseteq [n]$. It suffices to show that $i \in \Fix(e) \cup \Tr(e) \cup \Tl(e)$ for any $i \in [n]$.
When $e_{i-1}>e_i>e_{i+1}$, we consider two cases. If $i=1$, then $e_1>e_2$, which is impossible.
If $i>1$, then $e_{i-1}>e_i>e_{i+1}$ forms pattern $210$, a contradiction. Hence, the case
$e_{i-1}>e_i>e_{i+1}$ does not exist.
All the other cases are presented in the following table, where
 $(\rho_1,\rho_2)$ denotes  $e_{i-1} \, \rho_1 \, e_i$ and $e_{i} \, \rho_2 \, e_{i+1}$
and  the second row  indicates the set that $i$ belongs to.
\begin{table}[h]
\label{tab:ei}
\begin{centering}
\begin{tabular}{c|c|c|c|c|c|c|c}
\hline
    $(>,=)$ & $(\geq,<)$ & $(=,\geq)$ & $(<,>)$ &\multicolumn{2}{c|}{$(<,=)$} &  \multicolumn{2}{c}{$(<,<)$} \\
\hline
 \tabincell{c}{ $e_1=e_2=0$\\ $\Tr$ } & $\Fix$ & $\Tr$ & $\Fix$
&  \tabincell{c}{$e_i \neq e_{i-2}$\\  $\Fix$} &  \tabincell{c}{$e_i = e_{i-2}$\\  $\Tr$}
& \tabincell{c}{ $e_i = e_{i-2}$\\ $\Fix$} &  \tabincell{c}{$e_i \neq e_{i-2}$\\  $\Tl$}\\
\hline
\end{tabular}
\end{centering}
\end{table}
In view of the table, $i$ is an element of $\Fix(e)$, $\Tl(e)$ or $\Tr(e)$, as desired.
\end{proof}

{\bf The construction of $\gamma$.}
To construct $\gamma$, we first give a description of a map $\Gamma$ over $\I_n(100,210,201)$ by the following steps:
\begin{itemize}
  \item[1.]  For $i \in \Fix(e)$, keep $e_i$ unchanged.

  \item[2.]  For $i \in \Tr(e)$, from small to big, move $e_i$ to the right, increasing   the present value of $e_i$ by $1$ per element it passes and ending when it meets a greater element.

  \item[3.]   For $i \in \Tl(e)$, from big to small, move $e_i$ to the left, decreasing   the present value of $e_i$ by $1$ per element it passes and ending when it meets an equal element or it just passes a crucial element (of the present sequence) that is equal to its present value.
\end{itemize}
Let us define $\gamma=\Psi \circ \Gamma$.
For example, if we take $e=000033033346\in\I_{12}(100,210,201)$ from Example~\ref{exam:1},  then  $\gamma(e)$ can be computed  as in Fig.~\ref{alg:gam}.
\begin{figure}
\begin{center}
\begin{tikzpicture}[line width=0.8pt,scale=0.55]
\coordinate (O) at (0,0);

\path (O) node {$0$} ++(0.7,0) node {$0$} ++(0.7,0) node {$0$} ++(0.7,0) node[fill=gray!20]  {$0$} ++(0.7,0) node[fill=gray!20] {$3$} ++(0.7,0) node {$3$} ++(0.7,0) node[fill=gray!20] {$0$} ++(0.7,0) node {$3$} ++(0.7,0) node {$3$} ++(0.7,0) node[fill=gray!20] {$3$}++(0.7,0) node {$4$}++(0.7,0) node {$6$}
++(0,0.8) node[blue] {$11$};
\draw[blue, thick][->] (0,0.3) to [out=70,in=90] ++(1,0)  to [out=90,in=90] ++(0.7,0)  to [out=90,in=90] ++(0.7,0)  to [out=90,in=90] ++(0.7,0)  to [out=90,in=90] ++(0.7,0)
to [out=90,in=90] ++(0.7,0)
to [out=90,in=90] ++(0.7,0)
to [out=90,in=90] ++(0.7,0)
to [out=90,in=90] ++(0.7,0)
to [out=90,in=90] ++(0.7,0)
to [out=90,in=120] ++(0.7,0) ;

\path (O)++(0,-1.5)  node {$0$} ++(0.7,0) node {$0$} ++(0.7,0) node[fill=gray!20]  {$0$} ++(0,0.8) node[blue] {$2$}++(0,-0.8)++(0.7,0) node[fill=gray!20] {$3$} ++(0.7,0) node {$3$} ++(0.7,0) node[fill=gray!20] {$0$} ++(0.7,0) node {$3$} ++(0.7,0) node {$3$} ++(0.7,0) node[fill=gray!20] {$3$}++(0.7,0) node {$4$}++(0.7,0) node {$6$}
++(0.7,0) node[blue] {$11$};
\draw[blue, thick][->] (0,-1.2)  to [out=70,in=90] (1,-1.2)  to [out=90,in=120] (1.7,-1.2) ;

\path (O)++(0,-3)  node {$0$}  ++(0.7,0) node[fill=gray!20]  {$0$}
++(0,0.8) node[blue] {$1$}++(0,-0.8)++(0.7,0) node[blue] {$2$}++(0.7,0) node[fill=gray!20] {$3$} ++(0.7,0) node {$3$} ++(0.7,0) node[fill=gray!20] {$0$} ++(0.7,0) node {$3$} ++(0.7,0) node {$3$} ++(0.7,0) node[fill=gray!20] {$3$}++(0.7,0) node {$4$}++(0.7,0) node {$6$}
++(0.7,0) node[blue] {$11$};
\draw[blue, thick][->] (0,-2.7)  to [out=70,in=120] (1,-2.7);

\path (O)++(0,-4.5)   node[fill=gray!20]  {$0$}
++(0.7,0) node[blue] {$1$}++(0.7,0) node[blue] {$2$}++(0.7,0) node[fill=gray!20] {$3$} ++(0.7,0) node {$3$} ++(0.7,0) node[fill=gray!20] {$0$} ++(0.7,0) node {$3$} ++(0.7,0) node {$3$} ++(0.7,0) node[fill=gray!20] {$3$}++(0.7,0) node {$4$}++(0.7,0) node {$6$} ++(0,0.8) node[blue] {$9$}++(0,-0.8)
++(0.7,0) node[blue] {$11$};
\draw[blue, thick][->] (2.8,-4.2)  to [out=70,in=90] ++(1,0)
to [out=90,in=90] ++(0.7,0) to [out=90,in=90]++(0.7,0)
to [out=90,in=90]++(0.7,0) to [out=90,in=90] ++(0.7,0) to
[out=90,in=120] ++(0.7,0) ;

\path (O)++(0,-6)   node[fill=gray!20]  {$0$}
++(0.7,0) node[blue] {$1$}++(0.7,0) node[blue] {$2$}++(0.7,0) node[fill=gray!20] {$3$}  ++(0.7,0) node[fill=gray!20] {$0$} ++(0.7,0) node {$3$} ++(0.7,0) node {$3$} ++(0.7,0) node[fill=gray!20] {$3$}++(0.7,0) node {$4$}++(0.7,0) node {$6$}
++(0,0.8) node[blue] {$7$}++(0,-0.8)
++(0.7,0) node[blue] {$9$}
++(0.7,0) node[blue] {$11$};
\draw[blue, thick][->] (3.5,-5.7)  to [out=70,in=90] ++(1,0)
to [out=90,in=90] ++(0.7,0)  to [out=90,in=90] ++(0.7,0)
to [out=90,in=120] ++(0.7,0) ;

\path (O)++(0,-7.5)   node[fill=gray!20]  {$0$}
++(0.7,0) node[blue] {$1$}++(0.7,0) node[blue] {$2$}++(0.7,0) node[fill=gray!20] {$3$}  ++(0.7,0) node[fill=gray!20] {$0$} ++(0.7,0) node {$3$}  ++(0.7,0) node[fill=gray!20] {$3$}++(0.7,0) node {$4$}
++(0,0.8) node[blue] {$5$}++(0,-0.8)
++(0.7,0) node {$6$}
++(0.7,0) node[blue] {$7$}
++(0.7,0) node[blue] {$9$}
++(0.7,0) node[blue] {$11$};
\draw[blue, thick][->] (3.5,-7.2)  to [out=70,in=90] ++(1,0)
 to [out=90,in=120] ++(0.7,0);

\path (O)++(0,-9)   node[fill=gray!20]  {$0$}
++(0.7,0) node[blue] {$1$}++(0.7,0) node[blue] {$2$}++(0.7,0) node[fill=gray!20] {$3$}  ++(0.7,0) node[fill=gray!20] {$0$} ++(0.7,0) node[fill=gray!20] {$3$}
++(0,0.8) node[blue] {$3$}++(0,-0.8)
++(0.7,0) node {$4$}
++(0.7,0) node[blue] {$5$}
++(0.7,0) node {$6$}
++(0.7,0) node[blue] {$7$}
++(0.7,0) node[blue] {$9$}
++(0.7,0) node[blue] {$11$};
\draw[blue, thick][->] (5.6,-8.7)  to [out=90,in=70] ++(-1,0)
to [out=90,in=90] ++(-0.7,0)
 to [out=90,in=60] ++(-0.7,0);

 \path (O)++(0,-10.5)   node[fill=gray!20]  {$0$}
++(0.7,0) node[blue] {$1$}++(0.7,0) node[blue] {$2$}++(0.7,0) node[fill=gray!20] {$3$}  ++(0.7,0) node[fill=gray!20] {$0$}
++(0.7,0) node[blue] {$3$}
 ++(0.7,0) node[fill=gray!20] {$3$}
 ++(0,0.8) node[blue] {$3$}++(0,-0.8)
++(0.7,0) node {$4$}
++(0.7,0) node[blue] {$5$}
++(0.7,0) node[blue] {$7$}
++(0.7,0) node[blue] {$9$}
++(0.7,0) node[blue] {$11$};
\draw[blue, thick][->] (4.9,-10.2)  to [out=90,in=60] ++(-1,0);

\path (O)++(0,-12)   node[fill=gray!20]  {$0$}
++(0.7,0) node[blue] {$1$}++(0.7,0) node[blue] {$2$}++(0.7,0) node[fill=gray!20] {$3$}  ++(0.7,0) node[fill=gray!20] {$0$}
++(0.7,0) node[blue] {$3$}
 ++(0.7,0) node[blue] {$3$}
 ++(0.7,0) node[fill=gray!20] {$3$}
++(0.7,0) node[blue] {$5$}
++(0.7,0) node[blue] {$7$}
++(0.7,0) node[blue] {$9$}
++(0.7,0) node[blue] {$11$};

\path (O)++(0,-13)   node  {$0$}
++(0.7,0) node {$1$}++(0.7,0) node {$2$}++(0.7,0) node{$3$}  ++(0.7,0) node {$0$}
++(0.7,0) node {$0$}
 ++(0.7,0) node {$0$}
 ++(0.7,0) node {$0$}
++(0.7,0) node {$5$}
++(0.7,0) node {$7$}
++(0.7,0) node {$9$}
++(0.7,0) node {$11$};
\end{tikzpicture}
\caption{An example of the algorithm $\gamma$.\label{alg:gam}}
\end{center}
\end{figure}
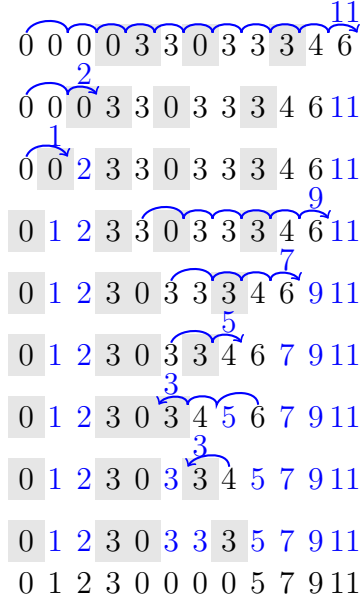

First of all, we need to show that $\gamma$ is well-defined, i.e., show that $\Gamma(e)\in\I_n(100,210,201)$. Given $e \in \I_n(100,210,201)$, assume that $\Tr(e)=\{r_1,r_2,\ldots,r_s\}$ and $\Tl(e)=\{l_1,l_2,\ldots,l_k\}$.
For $i \in [n] \setminus \Fix(e)$, define $M_i(e)$ to be the inversion sequence obtained by
\begin{itemize}
  \item   performing step $2$ in $\Gamma$ on $e_i$ if $i \in \Tr(e)$ or
  \item  performing step $3$ in $\Gamma$
on $e_i$ if $i \in \Tl(e)$.
\end{itemize}
Given a word $w=w_1w_2\cdots w_n$, we call $w_i$  a \emph{left-to-right maximum} if all  elements to the left of $w_i$ are smaller than $w_i$.
 We have the following two propositions.

\begin{proposition}\label{prop:gamma2}
  Suppose that $\bar{e}=M_{r_d}(e)(1 \leq d \leq s)$ where  $e_{r_d}$ ends  before $e_{j+1}$. Then
\begin{itemize}
  \item[(1)] $\bar{e} \in \I_n(100,210,201)$.
  \item[(2)] $e_i$ is a fixed element of $\bar{e}$ if and only if $e_i$ is a fixed element of $e$.

  \item[(3)] $\tr(\bar{e})=\tr(e)-1$. More precisely, $\Tr(\bar{e})=
      \{r_1,\ldots, r_{d-1}, {r_{d+1}}-1,\ldots,{r_c}-1, {r_{c+1}}, \ldots, {r_s} \}$, where $r_c \leq j <{r_{c+1}}$.

  \item[(4)] $\tl(\bar{e})=\tl(e)+1$.  More precisely, $\Tl(\bar{e})=\{{l_1} ,\ldots, {l_y}, {l_{y+1}}-1, \ldots, {l_x}-1, \, j, \, {l_{x+1}}, \ldots, {l_k} \}$, where
      $l_y < r_d <l_{y+1}$ and
      $l_x \leq j <{l_{x+1}}$.
  \item[(5)] $\pk(\bar{e})-\su(\bar{e})=\pk(e)-\su(e)$.
\end{itemize}
\end{proposition}

\begin{proof}
To prove item~(1), we first check that $\bar{e}_j=e_{r_d}+j-r_d$
is a left-to-right maximum of $\bar{e}$.
Firstly, we claim that $e_k \leq e_{r_d}$ for $0<k<r_d$. There are three cases to consider.
\begin{itemize}
\item  If $r_d=1$, it certainly holds.
\item If $r_d \in \Su(e)$, then $e_{r_d}=e_{r_d+1}$. Thus, $ e_k \leq e_{r_d}$ for $0<k<r_d$ follows from the fact that $e$ is $100$-avoiding.
\item If $r_d \notin \Su(e)$ and $r_d \neq 1$, then $e_{r_d-1}=e_{r_d}$. Similarly, since $e$ is $100$-avoiding,  we have $ e_k \leq e_{r_d}$ for $0<k<r_d-1$.
\end{itemize}
In all cases the claim is verified.
During $e_{r_d}$'s path to the right, we may check that
$e_{r_d}+x \geq e_{r_d+x+1}$ for $1 \leq x \leq j-r_d-1$.
Hence,  $\bar{e}_j$ is a left-to-right maximum of $\bar{e}$, which
 implies that $\bar{e}_j$ can not be the $0$ in an instance of pattern
 $100$  in $\bar{e}$.
Moreover, if there exist integers $1\leq j<l<h\leq n$
such that $\bar{e}_j \bar{e}_l\bar{e}_h$ forms a $100$-pattern of $\bar{e}$,
then $e_{j+1} e_l e_h=\bar{e}_{j+1} \bar{e}_l\bar{e}_h$ forms a $100$-pattern of $e$,
a contradiction. Therefore, $\bar{e}$ is $100$-avoiding. Similarly, we may prove that $\bar{e}$ is $210$-avoiding and $201$-avoiding. This completes the proof of item~(1).

In what follows, we wish  to prove items (2) to (5).
Obviously, $\bar{e}$
can be obtained from $e$ by deleting $e_{r_d}$  and inserting $e_{r_d}+j-r_d$ before $e_{j+1}$. To see the relations between
the set valued statistics of $e$ and $\bar{e}$, say $\Tr(e)$ and
$\Tr(\bar{e})$, we only need to care about the change of the roles
of $\bar{e}_{r_d-1}=e_{r_d-1}$, $\bar{e}_{r_d}=e_{r_d+1}$,
$\bar{e}_{j-1}=e_{j}$, $\bar{e}_{j}=e_{r_d}+j-r_d$ and $\bar{e}_{j+1}=e_{j+1}$, since other elements and their nearby stay unchanged.

For the case  with $\bar{e}_{j+1}=e_{j+1}$, we wish to prove that
$j+1 \in \Tl(e)$ (resp.~$ \Pk \setminus\Su(e)$)
if and only if $ j+1 \in \Tl(\bar{e})$ (resp.~$ \Pk \setminus \Su(\bar{e})$).

 On the one hand, we will show that $j+1$ is either in the set $\Tl(e)$
 or $ \Pk \setminus\Su(e)$, and
 if $j+1 \in \Tl(e)$ (resp.~$ \Pk \setminus\Su(e)$), then $ j+1 \in \Tl(\bar{e})$ (resp.~$ \Pk \setminus \Su(\bar{e})$).
Recall that  $\bar{e}_j$ is a left-to-right maximum of
$\bar{e}$. It follows that  $e_{j+1}$ is a left-to-right maximum of $e$.
Hence, we have $e_{j}<e_{j+1}$, $j+1 \notin \Su(e)$ and $j+1 \notin \Sf(e)$.
We consider two cases as follows.
\begin{itemize}
  \item If  $e_{j+1} < e_{j+2}$, then $j+1 \in \Tl(e)$.
        Since $\bar{e}_j\bar{e}_{j+1}\bar{e}_{j+2}=(e_{r_d}+j-r_d)e_{j+1}e_{j+2}$,
        then we deduce that $\bar{e}_j<\bar{e}_{j+1}<\bar{e}_{j+2}$, which implies that
        $j+1 \in \Tl(\bar{e})$.
  \item If  $e_{j+1} \geq e_{j+2}$, then $j+1 \in \Pk \setminus \Su(e)$. Notice that $\bar{e}_j<\bar{e}_{j+1} \geq \bar{e}_{j+2}$ and $\bar{e}_{j+1}$ is a left-to-right maximum of $\bar{e}$. Thus, $j+1 \in \Pk \setminus \Su(\bar{e})$ follows.
\end{itemize}

On the other hand, we show that $j+1$ is either in the set $\Tl(\bar{e})$ or $ \Pk \setminus\Su(\bar{e})$. Furthermore,
 if $j+1 \in \Tl(\bar{e})$ (resp.~$ \Pk \setminus\Su(\bar{e})$), then $ j+1 \in \Tl(e)$ (resp.~$ \Pk \setminus \Su(e)$). By the definition of the map $\Gamma$, we see that $\bar{e}_{j-1} \leq \bar{e}_j<\bar{e}_{j+1}$
 and so $j+1 \notin \Su (\bar{e})$ and $j+1 \notin \Sf (\bar{e})$.
 We consider two cases.
 \begin{itemize}
   \item  If $\bar{e}_j<\bar{e}_{j+1} \geq \bar{e}_{j+2}$,
   then $j+1 \in \Pk \setminus \Su (\bar{e})$.
    Since $e_je_{j+1}e_{j+2}=\bar{e}_{j-1}\bar{e}_{j+1}\bar{e}_{j+2}$,
     then $e_j < e_{j+1} \geq e_{j+2}$.
    Notice that $e_{j+1}$ is a left-to-right maximum of $e$.
    We deduce that $j+1 \in \Pk \setminus \Su(e)$.

   \item  If $\bar{e}_j<\bar{e}_{j+1} < \bar{e}_{j+2}$,
   then  $j+1 \in \Tl (\bar{e})$. Since $e_je_{j+1}e_{j+2}=\bar{e}_{j-1}\bar{e}_{j+1}\bar{e}_{j+2}$, we deduce that $e_j<e_{j+1}<e_{j+2}$. Furthermore, the fact $e_{j+1}$ is a left-to-right maximum of $e$ indicates that $j+1 \notin \Sf(e)$. It follows that  $j+1 \in \Tl(e)$.

 \end{itemize}

For the case with $\bar{e}_{r_d}=e_{r_d+1}$, we aim to show that  $e_{r_d+1}$ is a fixed (resp.~to-right, to-left) element of $e$
if and only if $\bar{e}_{r_d }$ is a fixed (resp.~to-right, to-left) element of $\bar{e}$.
We just present the proof of ``if'' direction here, and  the other direction follows directly in view of all the situations of $r_d$ in $\bar{e}$ indicated. To achieve this, we consider three cases.
\begin{itemize}
  \item If $r_d=1$, then $e_1=e_2=0$.
  When $e_3=0$, it can be  easily  checked that $2 \in \Tr(e)$.
  Since $\bar{e}_1=e_2=0$ and  $\bar{e}_2=e_3=0$, we have $1 \in \Tr(\bar{e})$. When $e_3>0$, then $e_1=e_2<e_3$ implies that $2 \in \Va(e) \subseteq \Fix(e)$. We claim that $1 \in \Va(\bar{e}) \subseteq \Fix(\bar{e})$.
  There are two cases. If $j=2$, namely, $e_1$ moves to the right and ends before $e_3$, then $+\infty > \bar{e}_1<\bar{e}_2=e_1+1$.  If $j>2$,
  then $+\infty > \bar{e}_1<\bar{e}_2=e_3$. In both cases the claim is verified.

  \item If $r_d \in \Su(e)$, then $e_{r_d-1}<e_{r_d-2}=e_{r_d}=e_{r_d+1}$.
      We may deduce that $e_{r_d+1}\leq e_{r_d+2}$. Otherwise,
      $e_{r_d-2} e_{r_d-1} e_{r_d+2}$ forms a pattern in $\{100, 201, 210\}$ of $e$.
      If $e_{r_d+1}=e_{r_d+2}$, then $r_d+1 \in \Tr(e)$.
      In this case we have $\bar{e}_{r_d-1}<\bar{e}_{r_d-2}=\bar{e}_{r_d}=\bar{e}_{r_d+1}$,
      and hence, $r_d \in \Su(\bar{e}) \subseteq \Tr(\bar{e})$.
      If $e_{r_d+1}<e_{r_d+2}$, then $r_d +1\in \Va(e) \subseteq \Fix(e)$. We claim that $r_d \in \Sf(\bar{e}) \subseteq \Fix(\bar{e})$. If $j=r_d+1$, then
      $\bar{e}_{r_d-1}<\bar{e}_{r_d-2}=\bar{e}_{r_d} \neq \bar{e}_{r_d+1}=e_{r_d}+1$. If $j>r_d+1$, then
      $\bar{e}_{r_d-1}<\bar{e}_{r_d-2}=\bar{e}_{r_d} \neq \bar{e}_{r_d+1}=e_{r_d+2}$. In both cases, $r_d \in \Sf(\bar{e})$ and the claim is verified.

  \item  If $r_d \notin \Su(e)$ and $r_d \neq 1$, then $e_{r_d-1}=e_{r_d} \geq e_{r_d+1}$. We consider two cases. If $e_{r_d} \geq e_{r_d+1} < e_{r_d+2}$, then $r_d+1 \in \Va(e)$.
     A routine check of the cases
      $j=r_d+1$ and $j>r_d+1$ shows that $r_d  \in \Va(\bar{e})$.
      If $e_{r_d} = e_{r_d+1} \geq e_{r_d+2}$, then $r_d+1 \in \Tr(e)$.  In this case, $\bar{e}_{r_d-1} = \bar{e}_{r_d} \geq \bar{e}_{r_d+1}$ implies that $r_d  \in \Tr(\bar{e})$.
\end{itemize}

Through a routine case-by-case discussion, we may similarly
 check that ${r_d-1} \in \Va(e)$ (resp.~$\Tr(e)$, $\Pk \setminus \Su(e)$)
if and only if ${r_d-1 } \in \Va(\bar{e})$ (resp.~$\Tr(\bar{e})$, $\Pk \setminus \Su(\bar{e})$);
$j \in \Va(e)$ (resp. $\Tl(e)$)
if and only if ${j-1 } \in \Va(\bar{e})$ (resp.~$\Tl(\bar{e})$).

For $\bar{e}_j=e_{r_1}+j-r_1$, we have $\bar{e}_{j-1}<\bar{e}_j<\bar{e}_{j+1}$. Since $\bar{e}_j$ is a
 left-to-right maximum of $\bar{e}$, then $j \notin \Sf(\bar{e})$.
 It follows that $j \in\Tl(\bar e)$.
Combining all the analysis above, items (2) to (5) follow.
\end{proof}

\begin{proposition}\label{prop:gamma3}
 Suppose that $\bar{e}=M_{l_d}(e)(1 \leq d \leq k)$ where  $e_{l_d}$ ends  after $e_{j-1}$. Then
\begin{itemize}
  \item[(1)] $\bar{e} \in \I_n(100,210,201)$.
  \vspace{0.1cm}
  \item[(2)] $e_i$ is a fixed element of $\bar{e}$ if and only if $e_i$ is a fixed element of $e$.
  \vspace{0.1cm}
  \item[(3)] $\tr(\bar{e})=\tr(e)+1$. More precisely, $\Tr(\bar{e})=\{{r_1},\ldots,{r_x},\, j, \,  {r_{x+1}+1}, \ldots, {r_y}+1, r_{y+1}, \ldots, r_s \}$, where $r_x < j \leq {r_{x+1}}$ and $r_y <l_d<r_{y+1}$.
  \vspace{0.1cm}
  \item[(4)] $\tl(\bar{e})=\tl(e)-1$.  More precisely, $\Tl(\bar{e})=\{{l_1},\ldots,{l_c},  {l_{c+1}+1}, \ldots, {l_{d-1}+1}, {l_{d+1}}, \ldots, {l_{k}} \}$, where $l_c < j \leq {l_{c+1}}$.
  \vspace{0.1cm}
  \item[(5)] $\pk(\bar{e})-\su(\bar{e})=\pk(e)-\su(e)$.
\end{itemize}
\end{proposition}

\begin{proof}
Firstly, it is not hard to see that $\bar{e}$ is $210$-avoiding and $201$-avoiding. Otherwise, for any $210$ or $201$-pattern of $\bar{e}$ containing $\bar{e}_j=e_{l_d}-l_d+j$, there is a $210$ or $201$-pattern of $e$ that replacing $\bar e_j$ by  the greater one of $e_{j-1}$ and $e_j$.

To prove that  $\bar{e}$ is $100$-avoiding, we  need to show the fact:
\begin{align*}
  (\star)\quad&\text{During the process of $e_{l_d}$'s path to the left, its present value is not smaller than } \\
&\text{the element it just passes.}
\end{align*}
Assume to the contrary that $x$ is the largest index with $j \leq x \leq l_d-2$
that $e_{l_d}$'s present value smaller than $e_x$ when it just passes $e_x$.
\begin{itemize}
 \item If $x=l_d-2$, then $e_x > e_{l_d}-2$ by the assumption. Besides, we have $e_x \neq  e_{l_d}-1$, otherwise
$e_{l_d}$'s path to the left will be ended before passing $e_x$. If $e_x =e_{l_d}$, then
$e_x e_{l_d-1}e_{l_d}$ will form a $101$-pattern, which contradicts with $l_d \in \Tl(e)$. If $e_x > e_{l_d}$, then
$e_x e_{l_d-1}e_{l_d}$ will form a $201$-pattern, which contradicts with $e$ is $201$-avoiding.
 \item If $x<l_d-2$, then $e_x > e_{l_d}-l_d+x$ by the assumption. Besides, we have $e_x \neq  e_{l_d}-l_d+x+1$, otherwise
$e_{l_d}$'s path to the left will be ended before passing $e_x$. If $e_x =e_{l_d}-l_d+x+2$ and $e_{x+2} =e_{l_d}-l_d+x+2$, then $e_x e_{x+1}e_{x+2}$
will form a $101$-pattern and
$e_{l_d}$'s path to the left will be ended before $e_{x+2}$, a contradiction. If $e_x =e_{l_d}-l_d+x+2$ and $e_{x+2} <e_{l_d}-l_d+x+2$, then $e_x e_{x+1}e_{x+2}$ will form
a pattern in $\{100,210,201\}$,  a contradiction. Hence, $e_x  \neq e_{l_d}-l_d+x+2$. If $e_x  > e_{l_d}-l_d+x+2$, then $e_x e_{x+1}e_{x+2}$ will form
 a $210$-pattern or a $201$-pattern,  a contradiction.
\end{itemize}
In both cases, there is no value for $e_x$. Hence, the assumption is not true and the fact $(\star)$ is confirmed.

Suppose that $\bar{e}$ contains a $100$-pattern. Since $e$ is
$100$-avoiding, any instance of pattern $100$ of $\bar{e}$
contains $\bar{e}_j$.
Moreover, $\bar{e}_j$ can not play the role of $1$, otherwise replacing $\bar{e}_j$ by  the greater one of $e_{j-1}$ and $e_j$, we will obtain a $100$-pattern of $e$, a contradiction.
If $\bar{e}_j$ is the $0$ in an instance of pattern $100$ of $\bar{e}$,
 then there exists an integer $l <j-1$ such that $\bar{e}_l \bar{e}_{j-1} \bar{e}_{j}$ or $\bar{e}_l \bar{e}_{j} \bar{e}_{j+1}$
forms a $100$-pattern. By fact $(\star)$, we see that $e_j \leq \bar{e}_j$. Notice that $\bar{e}_l=e_l$ and $\bar{e}_{j-1}=e_{j-1}$. Then, $e_l e_{j-1} e_j$ forms a  pattern in $\{100,210\}$,
or $e_l e_{j} e_{j+1}$ forms a pattern in $\{100,201\}$ of $e$,
a contradiction.
Hence, we deduce that $\bar{e}$ is  $100$-avoiding.
This completes the proof of item $(1)$.

Now, we wish  to prove items (2) to (5).
Similarly, we need only care about the change of the roles
of $\bar{e}_{j-1}=e_{j-1}$, $\bar{e}_{j}=e_{l_d}-l_d+j$,
$\bar{e}_{j+1}=e_{j}$, $\bar{e}_{l_d}=e_{l_d-1}$ and $\bar{e}_{l_d+1}=e_{l_d+1}$. Detailed proofs  of the cases
$\bar{e}_{j-1}=e_{j-1}$, $\bar{e}_{l_d}=e_{l_d-1}$ and $\bar{e}_{j}=e_{l_d}-l_d+j$  are included here, while the outline of the
proofs of the remaining ones are presented with the
analogous  details omitted.

For the case with $\bar{e}_{j-1}=e_{j-1}$, we may assume that $j>1$. We aim to show that
$j-1 \in \Tr(e)$ (resp.~$\Su(e)$, $\Pk \setminus \Su(e)$, $\Va(e)$)
if and only if
$j-1 \in \Tr(\bar{e})$ (resp.~$\Su(\bar{e})$, $\Pk \setminus \Su(\bar{e})$, $\Va(\bar{e})$). Analysis of the ``only if'' direction is given below with the proof of the ``if'' direction indicated.
\begin{itemize}
  \item If $e_{j-1} \geq e_j$, then $e_{l_d}$ ends after
   $e_{j-1}$ and $e_{j-1}=\bar{e}_j=e_{l_d}-l_d+j$.
   When $j=2$, obviously, $e_1=e_2=0$. It is easy to check that $1 \in \Tr(e)$ and $1 \in \Tr(\bar{e})$.
   When $j > 2$ and $e_{j-2}=e_{j-1}$, we have $j-1 \in \Tr(e)$. Since $\bar{e}_{j-2}=\bar{e}_{j-1}=\bar{e}_{j}$, we see that
   $j-1 \in \Tr(\bar{e})$.
   When $j > 2$ and $e_{j-2}<e_{j-1}$, we consider two cases.
   If $j-1 \in \Su(e)$, then $e_{j-3}e_{j-2}e_{j-1}$ forms
     a $101$-pattern of $e$.
     Recalling that $\bar{e}_{j-1}=\bar{e}_j$,  $j-1 \in \Su(\bar{e})$ follows.
   If $j-1 \in \Pk \setminus \Su(e)$, then $j-1 \in \Pk \setminus \Su(\bar{e})$ also follows from the fact $\bar{e}_{j-1}=\bar{e}_j$.

  \item If $e_{j-1} < e_j$, then $e_{l_d}$ ends just before
  the crucial element $e_j$ of $e$.
      It follows that $e_{j-1}<e_{j-2}=e_j$, which implies
      $j-1 \in \Va(e)$. On the other hand, since $\bar{e}_{j-1}=e_{j-1}$, $\bar{e}_{j-2}=e_{j-2}$ and
 $\bar{e}_j=e_{l_d}-l_d+j=\bar{e}_{j+1}$, we may deduce that
 $j-1 \in \Va(\bar{e})$.
\end{itemize}

For the case with $\bar{e}_{l_d}=e_{l_d-1}$,
we wish to prove that
$l_d-1 \in \Fix(e)$ (resp.~$\Tl(e)$) if and only if
$l_d  \in \Fix(\bar{e})$ (resp.~$\Tl(\bar{e})$).
We just give the proof of the ``if'' direction with the other direction  indicated.
Since $l_d \in \Tl(e)$, then  $e_{l_d-1}<e_{l_d}< e_{l_d+1}$ holds.
We consider the following three cases.
\begin{itemize}
  \item  If $l_d=2$, it can be easily checked that $1 \in \Va(e)$ and  $2 \in \Va(\bar{e})$.
  \item  If $l_d>2$ and $e_{l_d-2} \geq e_{l_d-1}$, then $l_d-1 \in \Va(e)$ follows from the fact that $e_{l_d-1}<e_{l_d}$.
      If $j=l_d-1$, then $\bar{e}_{l_d-2}=\bar{e}_{l_d-1}=e_{l_d}-1$.
      Recall that $e_{l_d}< e_{l_d+1}$, thus we have
      $\bar{e}_{l_d-1}\geq \bar{e}_{l_d}< \bar{e}_{l_d+1}$.
      If $j<l_d-1$, then $\bar{e}_{l_d-1}=e_{l_d-2}$. Clearly, we have $\bar{e}_{l_d-1}\geq \bar{e}_{l_d}< \bar{e}_{l_d+1}$.
     Therefore, for both cases we deduce that
     $l_d \in \Va(\bar{e})$.

  \item  If $l_d>2$ and $e_{l_d-2} < e_{l_d-1}$, we consider two subcases.
      \begin{itemize}
        \item If $l_d-1 \in \Tl(e)$, by routine check of the cases of $j=l_d-1$ and  $j<l_d-1$, we may deduce that $l_d \in \Tl(\bar{e})$.
        \item If $l_d-1 \in \Sf(e) \subseteq \Fix(e)$, there are two cases.
       When $j=l_d-1$, we have  $\bar{e}_{l_d-1}=e_{l_d}-1=e_{l_d-1}=\bar{e}_{l_d}$. Since $\bar{e}_{l_d}< \bar{e}_{l_d+1}$, we see that
       $l_d  \in \Va(e) \subseteq \Fix(e)$.
        When $j<l_d-1$, then $e_{l_d-3}e_{l_d-2}e_{l_d-1}$ forms a
        $101$-pattern of $e$. It follows that $\bar{e}_{l_d-1}<\bar{e}_{l_d-2}=\bar{e}_{l_d}<\bar{e}_{l_d+1}$,
        which means
        $l_d  \in \Sf(\bar{e}) \subseteq \Fix(\bar{e})$.
      \end{itemize}
\end{itemize}

For the set of $\bar{e}$ which $j$ belongs to, we analysis as follows. Notice that $\bar{e}_{j-1}\bar{e}_{j}\bar{e}_{j+1}=e_{j-1}(e_{l_d}+j-l_d)e_j$.
By definition of the map $\Gamma$, we need to deal with two cases below.
\begin{itemize}
  \item  If $e_{j-1} \geq e_{j}$, then it is not hard to check that
  $\bar{e}_{j-1}=\bar{e}_{j}\geq \bar{e}_{j+1}$. Hence, $j \in \Tr(\bar{e})$.
  \item If $e_{j-1} <e_{j-2}=e_{j}$, then we may deduce that
  $\bar{e}_{j-1} <\bar{e}_{j-2}=\bar{e}_{j}=\bar{e}_{j+1}$.
  It implies that $j \in \Su(\bar{e}) \subseteq \Tr(\bar{e})$.
\end{itemize}

Similarly, we may prove that $e_j$ is a fixed  (resp.~to-right) element of $e$ if and only if $\bar{e}_{j+1}$ is a fixed (resp.~to-right) element of $\bar{e}$; $l_d+1 \in \Tl(e)$ (resp.~$\Pk \setminus \Su (e)$)  if and only if $l_d+1 \in \Tl(\bar{e})$ (resp.~$\Pk \setminus \Su (\bar{e})$). In view of all the discussions above, we
complete the proof of items $(2)$ to $(5)$.
\end{proof}

\begin{lemma}\label{lem:welldefined}
The map $\Gamma:\I_n(100,210,201)\rightarrow\I_n(100,210,201)$ is well-defined.
\end{lemma}
\begin{proof}
Assume that $t=\Gamma(e)$ for $e \in \I_n(100,210,201)$.
By  Proposition \ref{prop:gamma1}, the role of each element
of  $e$ is defined. Obviously, the ending condition of step $2$ in $\Gamma$ is
well-defined. While, the fact $(\star)$ confirms the  rationality
 of the ending condition of step $3$. Furthermore,
by item~$(1)$ in Proposition~\ref{prop:gamma2} and
item~$(1)$ in Proposition~\ref{prop:gamma3}, we see that $t \in \I_n(100,210,201)$. The proof now is completed.
\end{proof}

  Lemma~\ref{lem:welldefined} asserts  that $\gamma$ is well-defined.
 In the following, we wish to show that $\gamma$ is a bijection.
 It suffices to prove that $\Gamma$ is an involution over $\I_n(100,210,201)$. Then,  the fact $\Gamma \circ \Psi^{-1}$ is the inverse of $\gamma$ follows directly. We need the following Lemma.


\begin{lemma}\label{lem:order}
Suppose that $e \in \I_n(100,210,201)$.
For any $a,b \in [n] \setminus \Fix(e)$ and $a<b$, we have
\begin{equation*}
  M_a\circ M_b(e)=M_b\circ M_a(e).
\end{equation*}
\end{lemma}

\begin{proof}
Firstly, by item (1) in Proposition  \ref{prop:gamma2} and item (1) in Proposition ~\ref{prop:gamma3} we see that $M_a(e) \in \I_n(100,210,201)$ and $M_b(e) \in \I_n(100,210,201)$. Hence, both $M_a\circ M_b$ and $M_b\circ M_a$ are well-defined.

 Let $\bar{e}=M_b\circ M_a(e)$.
 Suppose that $e_a$ and $e_b$ end at  positions
$p_a$ and $p_b$ of $\bar{e}$, respectively.
 It suffices to prove the following four properties.
\begin{itemize}
  \item [(I).] If $a,b \in \Tr(e)$, then $M_a \circ M_b(e)=M_b \circ M_a(e)$.
  \item [(II).] If $a,b \in \Tl(e)$, then $M_a \circ M_b(e)=M_b \circ M_a(e)$.
  \item [(III).]If $a \in  \Tr(e)$ and $b \in  \Tl(e)$,
  then  $M_a \circ M_b(e)=M_b \circ M_a(e)$.
 \item [(IV).] If $a \in  \Tl(e)$ and $b \in  \Tr(e)$,
 then $M_a \circ M_b(e)=M_b \circ M_a(e)$.
\end{itemize}

To prove (I), we need to show that either  $p_a >p_b$ or $p_a < b-1$ first.
We claim that $e_{r_1} \leq e_{r_2} \leq \cdots \leq e_{r_s}$ with $\Tr(e)=\{r_1,r_2, \ldots, r_s\}$.
For  each ${r_l} \in \Tr(e)$,   we have $e_{r_l}=e_{r_l+1}$
or $e_{r_l-1}=e_{r_l}$. Since $e$ avoids $100$-pattern, we deduce that $e_{x} \leq e_{r_l}$ with $x <r_l$ in both cases. The claim is verified. Further, we consider the following two cases:
\begin{itemize}
\item If $b \in \Su(e)$, then $e_{b-1}<e_{b-2}=e_{b}=e_{b+1}$.
There are two cases.
If $a=b-2$, then the present value of $e_{a}$  is greater than
$e_{b}$ by $2$ when it just passes $e_{b}$. Thus, the present value of $e_{a}$  is always greater than the present value of $e_{b}$ by $2$ when they pass the same element during
$e_{a}$'s  path to the right in $e$ and $e_{b}$'s path to the right in $M_a(e)$.  By
 the ending condition of the to-right elements defined in $\Gamma$,
 we deduce $p_a >p_b$.
When $a<b-2$, namely, $b-a-3 \geq 0$, we consider two subcases. If $e_{a}+b-a-3 < e_{b-2}$, then $e_{a}$
ends on the left of $e_{b-2}$. It follows that $p_a <b-1$.
Otherwise, assume that $e_{a}$
ends on the right of $e_{b-2}$.  Then,
$e_{a}+b-a-3 \geq e_{b-2}$ holds. Therefore, the present value of $e_{a}$  is always greater than that of $e_{b}$ by $3$ when they pass the same element during
$e_{a}$ and $e_{b}$'s paths to the right in $e$ and $M_a(e)$, respectively. Hence, we have $p_a >p_b$.

\item If $b \notin \Su(e)$, then $e_{b-1}=e_{b}$.
When $a=b-1$, the present value of $e_{a}$  is greater than
$e_{b}$ by $1$ when it just passes $e_{b}$. Thus, the present value of $e_{a}$ is always greater than that of $e_{b}$ by $1$ when they pass the same element during $e_{a}$ and $e_{b}$'s paths to the right in $e$ and $M_a(e)$, respectively. Hence, we have $p_a >p_b$.
When $a<b-1$, namely, $b-a-2 \geq 0$, we consider two subcases. If $e_{a}+b-a-2 < e_{b-1}$, then $e_{a}$
ends on the left of $e_{b-1}$. This implies that $p_a <b-1$.
Otherwise, we assume that $e_a$ ends on the right of $e_{b-1}$.
 It follows that $e_{a}+b-a-2 \geq e_{b-1}$. Thus the present value of $e_{a}$ is always greater than that of $e_{b}$  by $2$ when they pass the same element during $e_{a}$ and $e_{b}$'s path to the right in $e$
 and $M_a(e)$, respectively. Hence, we have $p_a >p_b$, as desired.
\end{itemize}

Now, we wish to prove  that
$M_b\circ M_a(e)=M_a\circ M_b(e)$ for $a,b \in \Tr(e)$.
 If $p_a<b-1$, then
it is obvious. If $p_a>p_b$,  then by item $(3)$ in Proposition \ref{prop:gamma2}
we see that $a\in \Tr(M_b(e))$ and $b\in \Tr(M_a(e))$. Moreover,  the elements $e_b$ passes
in $e$ are exactly the elements $e_b$ passes
in $M_a(e)$. Furthermore,  the present value of $e_a$
 during $e_a$'s path to the right in $e$
 is greater than the present value of $e_b$  during $e_b$'s path to the right in $M_a(e)$ at least by $1$ when they pass the same element.
 Therefore, the present value of $e_a$
 during $e_a$'s path to the right in $M_b(e)$
 is not smaller than the present value of $e_b$  during $e_b$'s path to the right in $e$ when they pass the same element in $\{e_{b+1}, e_{b+2},\ldots, e_{p_b+1}\}$. Thus, $e_a$ will not stop just before the element in
 $\{e_{b+1}, e_{b+2},\ldots, e_{p_b+1}, \bar{e}_{p_b}\}$ of $M_b(e)$.
  This means that the present value of $e_a$ during $e_a$'s path to the right of $e$ equals to the present value of $e_a$ during $e_a$'s path to the right of $M_b(e)$ when it passes the same element in
  $\{e_{p_b+2}, \ldots, e_{p_a}\}$. Thus, the number of the elements
  $e_a$ passes in $e$ equals to the number of the elements
  $e_a$ passes in $M_b(e)$. In all, the change of the position of
  $e_b$ brings no difference during $e_a$'s path to the right and
  we deduce that $M_b\circ M_a(e)=M_a\circ M_b(e)$.
The proof of (I) is completed.

To prove (II), we need to show that either $p_b <p_a$ or $1+a<p_b$. Firstly, we claim that $e_{l_x}$ $( 1 \leq x \leq k)$
is a left-to-right maximum of $e$. Since $l_x \in \Tl(e)$,
then $e_{l_{x}-1}<e_{l_x}<e_{l_x+1}$.  If there is some $y$ such that $y<l_x$ and
$e_y>e_{l_x}$, then $e_y e_{l_{x}-1}e_{l_x}$ forms a $201$-pattern, a contradiction. If there is some $y$ such that $y<l_x$ and
$e_y=e_{l_x}$, then we have $e_{l_x-2}=e_{l_x}$. Otherwise, $e_{l_x-2} e_{l_x-1}e_{l_x}$ forms a $201$-pattern or $e_y e_{l_x-2} e_{l_x-1}$ forms a $210$-pattern. Hence, we deduce that $l_x \in \Sf(e)$, which contradicts
to the fact that $l_x \in \Tl(e)$. The claim is verified.
Secondly, we show that $e_b$ can not stop just after $e_{a-1}$ in $M_b \circ M_a(e)$, i.e., $p_b\neq a+1$.
Since $a \in \Tl(e)$, then $e_{a-1}<e_{a}<e_{a+1}$.
It follows that $e_{a+1}$ is a left-to-right maximum and can not be the crucial element of $M_a(e)$.
Notice that  $e_{a-1}$ and $e_{a+1}$ are adjacent in $M_a(e)$.
If $e_b$ moves to the left and stop after $e_{a-1}$,  then
by Fact $(\star)$ we deduce that $e_{a-1}\geq e_{a+1}$. This contradicts with the fact that $e_{a-1}< e_{a+1}$, as desired.

We proceed to show that either $1+a<p_b$ or $p_b <p_a$.  As $p_b\neq a+1$, it suffices to show that if  $p_b\leq a$ then $p_b <p_a$. If  $p_b\leq a$, then $e_b-b+a+2>e_{a+1}$ holds, which implies that $e_b-b+a > e_{a}-1$ (as $e_{a}<e_{a+1}$).
     After $e_b$ passing elements $e_{b-1},\cdots,e_{a+1}, e_{a-1}$,
     the present value of $e_b$ is $e_b-b+a$.
     After $e_a$ passing the element $e_{a-1}$,
     the present value of $e_a$ is $e_a-1$.
     Since $e_b-b+a > e_{a}-1$,  the present value of
     $e_b$ is always greater than that of $e_a$ when they pass
     the same element during
     $e_a$ and $e_b$'s path to the left. We claim that $p_b< p_a$.
     Assume to the contrary that $p_b > p_a$, we consider two cases. If $\bar{e}_{p_b+1}$ is a crucial element of $M_a(e)$,
     then after passing $\bar{e}_{p_b+1}=e_{p_b-1}$ in $e$, the present value
     of $e_a$ is smaller than $e_{p_b-1}$. This contradicts with Fact $(\star)$.
     Otherwise, if $\bar{e}_{p_b-1}=\bar{e}_{p_b}$, then  after passing $\bar{e}_{p_b-1}=e_{p_b-2}$ in $e$, the present value
     of $e_a$ is smaller than $e_{p_b-2}$. This also contradicts with Fact $(\star)$. The claim is verified.

Now, we can prove  that
$M_b\circ M_a(e)=M_a\circ M_b(e)$ for $a,b \in \Tl(e)$.
If $p_b>a+1$, then it is obvious.
If $p_b < p_a$, then by item $(4)$ in Proposition \ref{prop:gamma3}
we see that $a\in \Tl(M_b(e))$ and $b\in \Tl(M_a(e))$. Moreover,
 the elements $e_a$ passes
in $e$ are exactly the elements $e_a$ passes
in $M_b(e)$. We claim that
the number of elements $e_b$ passes in $e$ equals to
the number of elements $e_b$ passes in $M_a(e)$.
The above analysis indicates  that  $e_b-b+a-1 \geq e_a-1$. This means that the present value of $e_b$ during $e_b$'s path
to the left  in $e$ is always not smaller than
 the present value of $e_a$ during $e_a$'s path
to the left  in $e$ when they passes the same element in
$\{e_{p_a-1},\ldots,e_{a-1}\}$. Thus, $e_b$ will not stop just after
an element in $\{e_{p_a-1},\ldots, e_{a-1}\}$ during its path to the left  in $e$.
It follows that the present value of $e_b$ during $e_b$'s path
to the left  of $e$ equals to the present value of $e_b$ during $e_b$'s path
to the left  of $M_a(e)$ when it passes the same element in
$\{e_{p_b},\ldots, e_{p_a-2}\}$.
Therefore, the change of the position of $e_a$ bring no difference
during $e_b$'s path to the left.
 The claim is verified and
we complete the proof of (II).

To prove (III), we wish to show that $a  \leq p_b <b \leq p_a$ or $a<p_a<p_b<b$ or $p_b \leq a<p_a <b$. We consider the following cases.
\begin{itemize}
  \item  If $p_a \geq b$, then we aim to prove that $a  \leq p_b <b \leq p_a$. Firstly, we claim that $e_b-b+a+1\leq e_a$.
      When $p_a=b$, then we have $e_b +1 \leq e_a+b-a<e_{b+1}$.
      Hence, we may deduce that $e_b-b+a+1\leq e_a$.
      When $p_a>b$, then we have $e_a+b-a \geq e_{b+1}$.
      This implies that  $e_b-b+a+1\leq e_a$. The claim is verified.
      Now, we proceed to show that $p_b \geq a$.
        Assume to the contrary that  $p_b<a$, then we see that
      $e_b$ moves to the left in $M_a(e)$ and does not stop after $e_{a-1}$.
      By considering two cases of $a \in \Su(e)$ and $a \notin \Su(e)$, we may
      always deduce that
      $e_b-b+a\geq  e_{a}$, which contradicts with the fact that $e_b-b+a+1\leq e_a$. Hence, we deduce that $p_b \geq a$.
  \item 
  If $p_a <b$, then there are two subcases.
  \begin{enumerate}
  \item If $p_a<p_b$, then $a<p_a<p_b<b$.  Moreover, we claim that $p_a+1 \neq p_b$.
  Assume to the contrary that  $p_a+1 =p_b$,  then $e_b$ stops just between $\bar{e}_{p_a}$ and  $e_{p_a+1}$ in $M_a(e)$. This is impossible by considering the facts that
   $\bar{e}_{p_a}<e_{p_a+1}$ and $e_{p_a+1}$ is a left-to-right maximum of $M_a(e)$. The claim is verified.
  \item If $p_a>p_b$, then we need to show that $p_b\leq a$. Since $e_b$ stops on the left of $\bar{e}_{p_a}$,  after passing the elements in
   $\{\bar{e}_{p_a}, e_{p_a}, \ldots, e_{b-1}\}$ in $M_a(e)$
   we have $e_b-b+p_a-1 \geq \bar{e}_{p_a}$ by Fact ($\star$).
   It follows that $e_b-b+p_a -(i+1)\geq \bar{e}_{p_a}-i$ $(1 \leq i \leq p_a-a-2)$, which means that
   the present value of $e_b$ after passing  elements of
   $\{e_{a+2}, \ldots, e_{p_a-1}\}$ in
   $M_a(e)$ is always not smaller than that
   of $e_a$ before passing the same element of
   $\{e_{a+2}, \ldots, e_{p_a-1}\}$ in
    $e$, respectively.  Based on this, we see that $e_b$ will not stop on the right of $e_{a+1}$ in  $M_a(e)$.
    Otherwise,  if $e_b$ stops because of coming  across an equal element  in front, then  $e_a$'s path to the right in $e$ would stop before the same element that $e_b$ stops after, a contradiction; if $e_b$ stops because of just passing an   equal crucial element, then  $e_a$'s path to the right in $e$ would stop before the other  element in pattern $101$ that equals to the crucial elements,  a contradiction. It follows that $p_b \leq a $.
    \end{enumerate}
\end{itemize}

Now, we are ready to show that $M_a \circ M_b(e)=M_b \circ M_a(e)$ for $a \in  \Tr(e)$ and $b \in  \Tl(e)$. We need to deal with the above three cases separately:
\begin{itemize}
\item For the case $a<p_a<p_b<b$, it is obvious based on the fact $p_a+1\neq p_b$.
\item For the case
$p_b \leq a<p_a <b$, by item $(3)$ in Proposition \ref{prop:gamma3} we have $a \in \Tr(M_b(e))$. Moreover, it is easy to check that
the elements $e_a$ passes
in $e$ are exactly the elements $e_a$ passes
in $M_b(e)$. It suffices to show that
the number of elements $e_b$ passes in $e$ equals to the number
of elements $e_b$ passes in $M_a(e)$.
The present value of $e_b$
during $e_b$'s path to the right in $e$ equals to
the present value of $e_b$
during $e_b$'s path to the right in $M_a(e)$ when passing the
elements in $\{e_{p_a}, \ldots, e_{b-1}\}$.
We claim that
$e_b$ will not stop after elements in $\{e_{a+1}, \ldots, e_{p_a-1}\}$
during $e_b$'s path to the left in $e$. Otherwise, assume to the contrary that $e_b$  stops after $e_x$$(\,  a+1 \leq x \leq p_a-1)$.
 Recall that $e_b-b+p_a-1 \geq \bar{e}_{p_a}$.
 Since $\bar{e}_{p_a}=e_a+p_a-a-1$, we deduce that $e_b-b \geq e_a-a$.
We consider two cases.
\begin{enumerate}
  \item If $e_x=e_b-b+x+1$, then we have $e_a-a+x +1\leq e_x$.
  Hence, $e_a-a+x-1 < e_x$, which means that $e_a$ will stop
  on the left of  $e_x$ during the path to the left in $e$. This contradicts with the fact that $p_a>x$.
  \item If $e_{x+1}=e_b-b+x+1$ and $e_{x-1}=e_{x+1}>e_x$,
   then we have $e_a-a+1+x \leq e_{x+1}$. Hence, we have $e_a-a+x < e_{x+1}$, which means that $e_a$ will stop
  on the left of  $e_{x+1}$ during the path to the left in $e$.
  Notice that $x \neq p_a-1$, since $e_{p_a}$ is not crucial.
  Then, the above fact contradicts with that $p_a >x+1$.
\end{enumerate}
The claim is verified.
Then, the present value of $e_b$
during $e_b$'s path to the right in $e$ equals to
the present value of $e_b$
during $e_b$'s path to the right in $M_a(e)$  when passing the
elements in $\{e_{p_b}, \ldots, e_{a-1}\}$, as desired.
\item For the case $a  \leq p_b <b \leq p_a$,
by item $(3)$ in Proposition \ref{prop:gamma3} we have $a \in \Tr(M_b(e))$ and by item $(4)$ in Proposition \ref{prop:gamma2} we have $b \in \Tl(M_a(e))$. Moreover, it can be easily checked that
the elements $e_b$ passes
in $e$ are exactly the elements $e_b$ passes
in $M_a(e)$.
Now, we need to show that
the number of elements $e_a$ passes in $e$ equals to the number
of elements $e_a$ passes in $M_b(e)$.
The present value of $e_a$
during $e_a$'s path to the right in $e$ equals to
the present value of $e_a$
during $e_a$'s path to the right in $M_b(e)$ when passing the
elements in $\{e_{a+1}, \ldots, e_{p_b}\}$ which may be empty.
Further, we claim that $e_a$ will not stop before elements in $\{e_{p_b+1}, \ldots, e_{b}\}$ during $e_a$'s path to the right in $e$. Otherwise, assume to the contrary that $e_a$ stop before $e_x$ ($p_b+1 \leq x \leq b$). It follows that $e_a +x-a-1 <e_x$.
Recall from the analysis above, we have $e_b-b+a+1 \leq e_a$.
Thus, we deduce that $e_b-b+x<e_x$, which means that
the present value of $e_b$ after $e_b$ just passing $e_x$ is smaller than $e_x$. This contradicts with Fact $(\star)$.  The claim is verified.
Then, the present value of $e_a$
during $e_a$'s path to the right in $e$ equals to
the present value of $e_a$
during $e_a$'s path to the right in $M_b(e)$   when passing the
elements in $\{e_{b+1}, \ldots, e_{p_a}\}$, as desired.
\end{itemize}

Notice that item (IV) is obvious based on the fact that $e_a$
and $e_b$ are not adjacent. We now complete the proof.
\end{proof}

\begin{lemma}\label{lem:Gamma}
Given $e \in \I_n(100,210,201)$, then we have $\Gamma^2(e)=e$. Moreover, we have
\begin{equation}\label{eq:jointdis}
  (\Dt,\tr,\tl,\pk-\su)e=(\Dt,\tl,\tr,\pk-\su)\Gamma(e).
\end{equation}

\end{lemma}

\begin{proof}
Assume that $b=\Gamma(e)$.
By item $(2)$ in Proposition \ref{prop:gamma2} and item $(2)$ in Proposition \ref{prop:gamma3}, we see that  $e_i$ is a fixed element of $b$ if and only if $e_i$ is a fixed element of $e$. Thus, step $1$
in $\Gamma$ is the inverse of itself.
Clearly, the process of an element's movement to the right is the inverse of its corresponding element's movement to the left and vice versa.
 Following from Lemma \ref{lem:order},
 the order of elements' movement can be exchanged freely.
 It follows that
 step $2$ and step $3$
in $\Gamma$ are  the inverse of steps $3$ and $2$, respectively. Hence, $\Gamma$ is an involution over $\I_n(100,210,201)$.

To show that $\Dt(e)=\Dt(b)$, firstly, we claim that $e_{i_1} < e_{i_2} < \cdots < e_{i_d}$ with $\Des(e)=\{i_1,i_2, \ldots, i_d\}$. Namely, let $i_x \in \Des(e)$ and  $i_y \in \Des(e)$ with $i_x< i_y$,
we have to show that $e_{i_x}< e_{i_y}$. Otherwise, if $e_{i_x}>e_{i_y}$,
then $e_{i_x}e_{i_y}e_{i_y+1}$ forms a $210$-pattern, a contradiction. If $e_{i_x}=e_{i_y}$, we consider three cases. When $e_{i_x+1}>e_{i_y+1}$, then $e_{i_x} e_{i_x+1} e_{i_y+1}$ forms a $210$-pattern.  When $e_{i_x+1}=e_{i_y+1}$, $e_{i_x} e_{i_x+1} e_{i_y+1}$ forms a $100$-pattern.
When $e_{i_x+1}<e_{i_y+1}$, then $e_{i_x} e_{i_x+1} e_{i_y+1}$ forms a $201$-pattern. Hence, $e_{i_x}=e_{i_y}$ is impossible as contradicting
with $e \in \I_n(100,210,201)$. The claim is verified.

Now, based on the fact that $\Gamma$ is an involution, it is enough to show that $\Dt(e) \subseteq \Dt(b)$. For $e_i \in \Dt(e)$, we deduce that $e_{i-1}\leq e_{i}$, otherwise $e_{i-1}e_{i}e_{i+1}$ is an instance of pattern $210$.
Besides, we have $e_{i+1}<e_{i+2}$, otherwise $e_i e_{i+1}e_{i+2}$ forms a $210$-pattern or  a $100$-pattern. It indicates that $i+1 \in \Va(e)$, and hence, $i+1 \in \Fix(e)$. We consider the following two cases.
\begin{itemize}
  \item If $e_{i-1}< e_i$, then it can be easily checked that $i \in \Pk(e) \setminus \Su(e)$. It follows that $i \in \Fix(e)$.
 Assume that $e'$ is obtianed form $e$ by performing steps $1$ and $2$ in $\Gamma$. Clearly, $e_i e_{i+1}$ is consecutive in $e'$. During the performance of
step $3$ on $e'$, elements that are equal to $e_i$ may be inserted between $e_i$ and  $e_{i+1}$, or not. This implies that $e_i \in \Dt(b)$, as desired.
  \item  If $e_{i-1}= e_i$, let $s$ be the integer such that $e_{i-s-1} \neq e_{i-s}=\cdots e_i $. We have $e_{i-s-1} <e_{i-s} $, otherwise $e_{i-s-1}e_{i-s}e_i$ forms a $100$-pattern. It is easily checked that   $i-s \in \Pk(e)$.
      Furthermore, we claim that  $i-s \notin \Su(e)$.
      Otherwise, assume to the contrary that $e_{i-s-1}<e_{i-s-2}=e_{i-s}$, then
      $e_{i-s-2} e_{i-s-1} e_{i+1}$ forms a $100$-pattern
      or a $210$-pattern. This contradicts with $e \in \I_n(100,210,201)$.
      Hence, we have $i-s \in \Fix(e)$. Assume that $e'$ is obtianed form $e$ by performing steps $1$ and $2$ in $\Gamma$. We see that $e_{i-s} e_{i+1}$ is consecutive in $e'$. During the performance of
step $3$ on $e'$, elements that are equal to $e_{i-s}$ may be inserted between $e_{i-s}$ and  $e_{i+1}$,  or not. Since $e_{i-s}=e_i$, then
$e_i \in \Dt(b)$, as desired.
\end{itemize}

The fact $(\tr,\tl,\pk-\su)e=(\tl,\tr,\pk-\su)b$ comes directly from items $(3)$ to $(5)$ in Proposition \ref{prop:gamma2} and items $(3)$ to $(5)$ in Proposition \ref{prop:gamma3}. This completes the proof.
\end{proof}

Now, we are ready to give a proof of Theorem \ref{thm:100101}.
\begin{proof}[{\bf Proof of Theorem \ref{thm:100101}}]
Given $e \in \I_n(100,210,201)$, let $b=\Gamma(e)$ and $t=\Psi(b)$.
It remains to show that
\begin{align}
  \label{eq:Dt}\Dt(e) &= \Dt(t), \\
  \label{eq:asc}\asc(e) &=n-1-\asc(t).
\end{align}

To prove (\ref{eq:Dt}), in view of Lemma \ref{lem:Gamma}, we need to show that $\Dt(b) = \Dt(t)$.
Given $b_i \in \Dt(b)$, we have $b_i > b_{i+1}$.  We claim that
there is no $s<k<i$ such that $b_s=b_i>b_{k}$.
Otherwise, if $b_{k}\geq b_{i+1}$, then $b_s b_{k} b_{i+1}$
will form a $210$-pattern or a $100$-pattern, a contradiction.
If $b_{k}< b_{i+1}$,  then $b_s b_{k} b_{i+1}$
will form a $201$-pattern, a contradiction. The claim is verified.
Hence, we have $t_i=b_i$. Then $t_{i+1} \leq b_{i+1}$
  indicates that $b_i \in \Dt(t)$. It follows that
$ \Dt(b) \subseteq \Dt(t)$.

Given $t_i \in \Dt(t)$, we have $t_i > t_{i+1}$.
There is no $s <k <i $ such that $t_s>t_k=t_{i}$,
otherwise, $t_s t_i t_{i+1}$ forms pattern $210$.
Hence, $b_i=t_i$.
Further, we claim that $b_{i+1}<b_i$. If $b_{i+1}=t_{i+1}$, this clearly holds. If $b_{i+1}>t_{i+1}$, then there exist integers $l<j<i+1$ such that $t_l t_j t_{i+1}$ forms pattern $100$ and $t_l$ is the maximal $1$ among such instances of $100$. We deduce that
$t_l<t_i$, if not, $t_l t_i t_{i+1}$ will form a $210$-pattern.
Hence, $b_{i+1}=t_l<t_i$. The claim is verified and
$\Dt(t)\subseteq \Dt(b)$ follows. The proof of (\ref{eq:Dt})
is completed.

In the following, we will prove (\ref{eq:asc}). By definition, it is easy to check that
\begin{align}\label{eq:ascexpand-e}
  \asc(e)&=\tl(e)+\va(e)+\sf(e)-1\\[3pt]
  \asc(b)&=\tl(b)+\va(b)+\sf(b)-1. \label{eq:ascexpand-b}
\end{align}
Since $\Psi$ change each $101$ to $100$, then $\asc$ decreases by $1$ after each change. Hence, we deduce that
\begin{align*}
  \asc(t)&=\asc(b)-\sf(b)-\su(b)\\[3pt]
   & =\tl(b)+\va(b)-\su(b)-1 \,\,\,[\text{by Equation~\eqref{eq:ascexpand-b}}]\\[3pt]
   &=\tl(b)+\pk(b)-\su(b)\,\,\,[\text{by Proposition~\ref{prop:gamma}}]\\[3pt]
   &=\tr(e)+\pk(e)-\su(e)\,\,\,[\text{by Equation ~\eqref{eq:jointdis}}]\\[3pt]
   &=n-\tl(e)-\va(e)-\sf(e)\,\,\,[\text{by Proposition~\ref{prop:gamma1}}]\\[3pt]
   &=n-1-\asc(e), \,\,\,[\text{by Equation~\eqref{eq:ascexpand-e}}]
\end{align*}
as desired.
\end{proof}

It should be noted that  the map $\Gamma$ can be restricted to
$\I_n(>,-,\geq)$.
which serves as a proof of  Corollary \ref{cor:intersect}.


\begin{proof}[{\bf Proof of  Corollary \ref{cor:intersect}}]
Recall that
 \[\I_n(>,-,\geq)=\I_n(>,-,>) \cap \I_n(>,\neq,\geq)=\I_n(100,101,210,201).\]
We wish to prove that $\Gamma$ is an involution over $\I_n(100,101,210,201)$.
Given such $e$ and let $t=\Gamma(e)$,
it suffices to show that $t \in \I_n(100,101,210,201)$.
By Lemma \ref{lem:Gamma}, we have $t \in \I_n(100,210,201)$.
Hence, to prove this lemma, we need to show that $t$ is
$101$-avoiding.

Given $\Tr(e)=\{r_1,r_2, \ldots, r_s\}$,
let $\bar{e}$ be obtained from $e$ by moving $e_{r_1}$ to the
 right, and ending before $e_{j+1}$. In the proof of item $(1)$ in
 Proposition \ref{prop:gamma2}, we see that
 $\bar{e}_j=e_{r_1}+j-r_1$ is a left-to-right maximum of $\bar{e}$.
Hence, $\bar{e}_j$ can not play the role of $0$ and the second $1$
in a $101$-pattern. If there are integers $l>d>j$ such that
$\bar{e}_j=\bar{e}_l>\bar{e}_d$, then $e_{j+1} e_l e_d$ will form a
$201$-pattern of $e$, a contradiction. Hence, $\bar{e}$ avoids $101$-pattern.
Given $\Tl(e)=\{l_1,l_2, \ldots, l_k\}$,
let $\bar{e}$ be obtained from $e$ by moving $e_{l_d}(1 \leq d \leq l_k)$ to the left, and ending after $e_{j-1}$.  Then,
$\bar{e}_{j-1}=\bar{e}_{j}$ or $\bar{e}_{j}=\bar{e}_{j+1}$.
Thus, any instance of $101$-pattern  of $\bar{e}$ containing $\bar{e}_{j}$ implies an instance of $101$-pattern  of $e$ containing $e_{j-1}$ or $e_j$, a contradiction.
 Hence, $\bar{e}$ avoids $101$-pattern.

By iteratively using the above two facts, we deduce that $t$ avoids the pattern $101$. This completes the proof.
\end{proof}

\section{Lehmer code, $b$-code, and $\gamma$-positivity}

\label{gam:posi}

\subsection{Applications of Lehmer code and Foata--Strehl action}
The Lehmer code $\Theta$ defined in the introduction has plenty of applications in enumerating pattern avoiding inversion sequences (see~\cite{auli2,MS,fl,ly}).
For $(>,-,\geq)$-avoiding inversion sequences, the following application of Lehmer code was proved in~\cite[Theorem~40]{MS}.
\begin{proposition}[Martinez and Savage]
\label{pro:MS40}
The Lehmer code $\Theta$ restricts to a bijection between $\SS_n(2134,2143)$ and $\I_n(>,-,\geq)$. Consequently,
\begin{equation}\label{2134,2143}
\sum_{\pi\in\SS_n(2134,2143)}t^{\Des(\pi)}=\sum_{e\in\I_n(>,-,\geq)}t^{\Asc(e)}.
\end{equation}
\end{proposition}

The binomial transformation of Fine's sequence can be defined by the algebraic generating function
$$
\frac{2}{1+x+\sqrt{1-6x+5x^2}}=1+x+2x^2+6x^3+21x^4+79x^5+311x^6+1265x^7+\cdots.
$$
 Mansour and Shattuck~\cite{Mansour0} proved that nine classes of permutations avoiding triples of $4$-letter patterns are enumerated by the binomial transformation of Fine's sequence, one of which is the class of $(2134,2143,2314)$-avoiding permutations.
\begin{proposition}
\label{pro:fine}
The Lehmer code $\Theta$ restricts to a bijection between $\SS_n(2134,2143,3124)$ and $\I_n(201,210,110,101,100)$. Consequently, the class $\I_n(201,210,110,101,100)$ is enumerated by the binomial transformation of Fine's sequence.
\end{proposition}
\begin{proof} As the mapping $\pi\mapsto \pi^{-1}$ sets up a bijection between $\SS_n(2134,2143,3124)$ and $\SS_n(2134,2143,2314)$, the second statement then follows from the first one and the aforementioned result of Mansour--Shattuck. It remains to show that $\Theta(\SS_n(2134,2143,3124))=\I_n(201,210,110,101,100)$.

Let $\pi\in\SS_n$ and $e=\Theta(\pi)\in\I_n$. Notice that $\I_n(>,-,\geq)=\I_n(201,210,101,100)$. If $e$ contains a pattern in $\{201,210,110,101,100\}$, then we distinguish two cases.
\begin{itemize}
\item If $e\notin\I_n(>,-,\geq)$, then by Proposition~\ref{pro:MS40}, $\pi$ contains at least one of the patterns $2134$ and $2143$.
\item Otherwise, $e\in\I_n(>,-,\geq)$ and $e$ contains the pattern $110$. Then, there exist indices $i<j-1$ such that $e_i=e_{i+1}>e_j$. It follows that $\pi_i<\pi_{i+1}<\pi_j$ and there must exist an index $a$ smaller than $i$ such that $\pi_j>\pi_a>\pi_{i+1}$ (since $e_{i+1}>e_j$). Thus $\pi_a\pi_i\pi_{i+1}\pi_j$ forms the pattern $3124$.
\end{itemize}

Conversely, suppose that $\pi$ contains a pattern in $\{2134,2143,3124\}$.
\begin{itemize}
\item If $\pi\notin\SS_n(2134,2143)$, then by Proposition~\ref{pro:MS40}, $e$ contains at least one pattern in $\{201,210,101,100\}$.
\item Otherwise, $\pi\in\SS_n(2134,2143)$ and  so $\pi$ contains a $3124$-pattern $\pi_i\pi_j\pi_k\pi_l$ such that $\pi_j<\pi_{j+1}$ (as $\pi$ avoids $2134$). Consider the set $A=\{a: j<a< l, \pi_a>\pi_l\}$.
\begin{enumerate}
\item If $A=\emptyset$, then $e_j\geq e_{j+1}$ and $e_j>e_l$.
\item If $A=\{a\}$, then $e_j>e_a$ and $e_j\geq e_l$.
\item If $A=\{a_1,a_2,\ldots\}$, then  $e_j>e_{a_1}$ and $e_j\geq e_{a_2}$.
\end{enumerate}
\end{itemize}
In either case, $e$ contains a pattern in $\{201,210,101,100,110\}$.
\end{proof}

\begin{figure}
\setlength{\unitlength}{1mm}
\begin{picture}(108,38)\setlength{\unitlength}{1mm}
\thinlines
\red{\put(12,15){\dashline{1}(1,0)(40,0)}
\put(12,15){\vector(-1,0){0.1}}}

\put(16,19){\dashline{1}(1,0)(32,0)}
\put(48,19){\vector(1,0){0.1}}

\put(40,27){\dashline{1}(-1,0)(-16,0)}
\put(40,27){\vector(1,0){0.1}}

\put(68,23){\dashline{1}(1,0)(16,0)}
\put(84,23){\vector(1,0){0.1}}

\put(100,7){\dashline{1}(-1,0)(-96,0)}
\put(4,7){\vector(-1,0){0.1}}

\put(0,3){\line(1,1){32}}\put(-2,0){$-\infty$}
\red{\put(52,15){\circle*{1.3}}\put(52.8,15){$3$}}
\put(16,19){\circle*{1.3}}\put(13,19){$4$}
\put(32,35){\circle*{1.3}}\put(29,35){$8$}

\put(32,35){\line(1,-1){24}}
\put(56,11){\circle*{1.3}}\put(57,9){$2$}
\put(24,27){\circle*{1.3}}\put(21,26.8){$6$}

\put(56,11){\line(1,1){20}}
\put(76,31){\circle*{1.3}}\put(77,31){$7$}
\put(68,23){\circle*{1.3}}\put(65,23){$5$}

\put(76,31){\line(1,-1){28}}\put(103,0){$-\infty$}
\put(100,7){\circle*{1.3}}\put(101,7){$1$}
\end{picture}
\caption{MFS-actions on $46832571\in\SS_8$.\label{mfs:action}}
\end{figure}
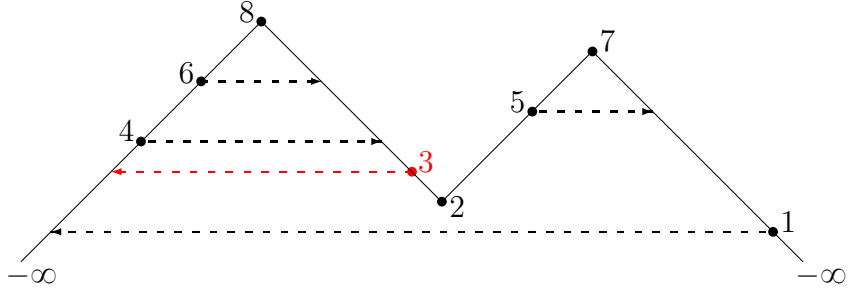

The Foata--Strehl action~\cite{fsh} (see also~\cite{Ath}) on permutations interprets combinatorially the $\gamma$-positivity of the Eulerian polynomials. It can be applied to give a neat proof of~\eqref{gam:kl} with the aid of identity~\eqref{2134,2143}. For $\pi\in\SS_n$ and $a\in[n]$, $\pi$ can be factorized as
$$
\pi=w_1w_2aw_3w_4,
$$
where $w_2$ (resp.~$w_3$) is the maximal contiguous interval (possibly empty) immediately to the left (resp.~right) of $a$ whose letters are all greater than $a$. The {\em Foata--Strehl action} $\varphi_a$ on $\pi$ is defined by
$$
\varphi_a(\pi)=w_1w_3aw_2w_4.
$$
For example, if $a=3$ and $\pi=46832571\in\SS_8$, then $w_1=\emptyset$, $w_2=468$, $w_3=\emptyset$ and $w_4=2571$. Thus, $\varphi_a(\pi)=34682571$; see the point $3$ in Fig.~\ref{mfs:action}. The {\em Modified Foata--Strehl action} (abbreviated as {\em MFS-action}) $\varphi_a'$ is defined by
$$
\varphi'_a(\pi):=
\begin{cases}
\varphi_a(\pi),&\text{if  $\pi_{i-1}<\pi_i<\pi_{i+1}$ or $\pi_{i-1}>\pi_i>\pi_{i+1}$, where $i=\pi^{-1}(a)$};\\
\pi,& \text{otherwise.}
\end{cases}
$$
Here we use the convention $\pi_0=\pi_{n+1}=-\infty$. See Fig.~\ref{mfs:action} for the visualization of the MFS-actions.

An index $i\in[n-2]$ is called a {\em double descent} of  $\pi\in\SS_n$ if $\{i,i+1\}\subseteq\Des(\pi)$. The following fundamental result of MFS-actions was proved in~\cite{kl}.

\begin{lemma}\label{act:gam}
 Suppose $\mathcal{S}\subseteq\SS_n$ is invariant under the MFS-action. Then,
$$
\sum_{\pi\in\mathcal{S}}t^{\des(\pi)}=\sum_{k=0}^{\lfloor(n-1)/2\rfloor}|\widetilde{\mathcal{S}_{n,k}}|t^k(1+t)^{n-1-2k},
$$
where $\widetilde{\mathcal{S}_{n,k}}:=\{\pi\in \mathcal{S}:\des(\pi)=k,\text{ $\pi$ has no double descents and $\pi_{n-1}<\pi_n$}\}$.
\end{lemma}

\begin{lemma}\label{inv:2134}
The set $\SS_n(2134,2143)$ is invariant under the MFS-action.
\end{lemma}
\begin{proof}
Suppose that $\pi\notin\SS_n(2134,2143)$, then $\pi$ contains subsequence $\pi_i\pi_j\pi_k\pi_l$ that is order
isomorphic
 to  $2134$ or $2143$, i.e., $\pi_j<\pi_i<\min\{\pi_k,\pi_l\}$. By the definition of $\varphi_a'(\pi)$, the letter $\pi_i$ in $\varphi_a'(\pi)$ is still appear to the left of $\pi_j$, while the letters  $\pi_k$ and $\pi_l$ in $\varphi_a'(\pi)$ are still appear to the right of $\pi_j$. Thus, $\varphi_a'(\pi)\notin\SS_n(2134,2143)$. The result then follows from the fact that $\varphi_a'$ is an involution on $\SS_n$ for any $a\in[n]$.
\end{proof}

\begin{proof}[{\bf Proof of Proposition~\ref{pro:kl}}]
By Lemmas ~\ref{act:gam} and ~\ref{inv:2134}, we have
$$
\sum_{\pi\in\SS_n(2134,2143)}t^{\des(\pi)}=\sum_{k=0}^{\lfloor(n-1)/2\rfloor}|\widetilde{\SS_{n,k}}(2134,2143)|t^k(1+t)^{n-1-2k}.
$$
The $\gamma$-positivity expansion~\eqref{gam:kl} then follows by applying  identity~\eqref{2134,2143}
\end{proof}

\subsection{The $b$-code and proof of Theorem~\ref{thm:gamcl}}
Baril and Vajnovszki \cite{Baril} gave another coding of permutations, called {\em $b$-code},
 which preserves a double Eulerian bistatistic. Recent enumerative applications of $b$-code have been found in~\cite{kl,kl2,fjlyz,ly} and our proof of Theorem~\ref{thm:gamcl} is another such instance.

Let us recall briefly the construction of the $b$-code. An  interval $[m,n]$ with $m <n$ is the set $\{x \in \mathbb{N} \colon m \leq x \leq n \}$.
A labeled interval is a pair $(I,l)$, where $I$ is an interval and
$l$ is a nonnegative integer.
Given $\pi=\pi_1 \pi_2 \cdots \pi_n \in \SS_n$ and an integer $i$ with
$0 \leq i <n$, let the $i$-th slice of $\pi$, $U_i(\pi)$,
 be a sequence of labelled intervals constructed recursively by the following process. Set $U_0(\pi)=([0,n],0)$. For $i \geq 1$, assume that $U_{i-1}(\pi)=(I_1,l_1),(I_2,l_2),\cdots,(I_k,l_k)$ is the $(i-1)$-th
slide of $\pi$ and $v$  is the index such that
$\pi_i \in I_v$, then $U_i(\pi)$ is constructed according to the following four cases.
\begin{itemize}
  \item If $\min(I_v)<\pi_i =\max(I_v)$, then $U_i(\pi)$ equals
  \[(I_1,l_1), \cdots , (I_{v-1},l_{v-1}), (J,l_{v+1}),(I_{v+1},l_{v+2}), \cdots , (I_{k-1},l_{k}),(I_{k},l_{k}+1),\]
  where $J=[\min(I_v),\pi_i-1]$.

 \item If $\min(I_v)<\pi_i <\max(I_v)$, then $U_i(\pi)$ equals
  \[(I_1,l_1), \cdots , (I_{v-1},l_{v-1})(H,l_{v}), (J,l_{v+1}),(I_{v+1},l_{v+2}), \cdots , (I_{k-1},l_{k}),(I_{k},l_{k}+1),\]
  where $H=[\pi_i+1, \max(I_v)]$ and $J=[\min(I_v),\pi_i-1]$.

  \item If $\min(I_v)=\pi_i <\max(I_v)$, then $U_i(\pi)$ equals
  \[(I_1,l_1), \cdots , (I_{v-1},l_{v-1})(H,l_{v}), (I_{v+1},l_{v+1}), \cdots , (I_{k-1},l_{k-1}),(I_{k},l_{k}+1),\]
  where $H=[\pi_i+1, \max(I_v)]$.

  \item If $\min(I_v)=\pi_i =\max(I_v)$, then $U_i(\pi)$ equals
  \[(I_1,l_1), \cdots , (I_{v-1},l_{v-1}), (I_{v+1},l_{v+1}), \cdots , (I_{k-1},l_{k-1}),(I_{k},l_{k}+1).\]
\end{itemize}
Let $b(\pi)=b_1 b_2 \cdots b_n \in\I_n$, where $b_i=l_v$
such that  $(I_v,l_v)$ is a labelled interval in the
$(i-1)$-th slice of $\pi$ with $\pi_i \in I_v$.

\begin{example}
For $\pi=6132547\in\SS_6$, we can compute the following slices:
\begin{align*}
&U_0(\pi)=([0,7],0);\\
&U_1(\pi)=([7,7],0),([0,5],1);\\
&U_2(\pi)=([7,7],0),([2,5],1),([0,0],2);\\
&U_3(\pi)=([7,7],0),([4,5],1),([2,2],2),([0,0],3);\\
&U_4(\pi)=([7,7],0),([4,5],1),([0,0],4);\\
&U_5(\pi)=([7,7],0),([4,4],4),([0,0],5);\\
&U_6(\pi)=([7,7],0),([0,0],6).
\end{align*}
Reading the labels from the above slices, we get  $b(\pi)=(0,1,1,2,1,4,0)$.
\end{example}

An interval $I$ is said to be {\em lower} than another interval $J$ if $\max(I)<\min(J)$; otherwise, $I$ is {\em higher} than $J$. The following basic properties  of $b$-code was  observed  in~\cite{kl2}.

\begin{lemma}\label{lem:dec}
Let $\pi\in\SS_n$  and $0\leq i<n$. If $U_{i}(\pi)=(I_1,\ell_1),(I_2,\ell_2),\ldots,(I_k,\ell_k)$,
then
\begin{enumerate}
\item[(i)] the interval $I_1,I_2,\ldots,I_k$ are in decreasing order whose labelings $\ell_1,\ell_2,\ldots,\ell_k$ are strictly increasing;
\item[(ii)] the labelings $ \ell_1,\ell_2,\ldots,\ell_{k-1}$ must appear as entries of $b(\pi)$ after its $i$-th entry.
\end{enumerate}
\end{lemma}

\begin{proposition}
The  $b$-code  restricts to a bijection between $\SS_n(24135,24153,42135,42153)$ and $\I_n(>,\neq,>)$. Consequently,
\begin{equation}\label{24135}
\sum_{\pi\in\SS_n(24135,24153,42135,42153)}t^{\Des(\pi)}=\sum_{e\in\I_n(>,\neq,>)}t^{\Asc(e)}.
\end{equation}
\end{proposition}
\begin{proof}
Let $\pi\in\SS_n$ and $e=b(\pi)\in\I_n$.
If $\pi$ contains an occurrence $\pi_i\pi_j\pi_k\pi_l\pi_m$ of one of the patterns in $\{24135,24153,42135,42153\}$, then in the $(k-1)$-th slice $U_{k-1}(\pi)$, the three letters $\pi_k$, $\pi_l$ and $\pi_m$ belong to three different intervals, among which the interval containing $\pi_k$ receives  the greatest labeling according to Lemma~\ref{lem:dec}~(i). Thus by Lemma~\ref{lem:dec}~(ii), the suffix $e_ke_{k+1}\cdots e_n$  contains a pattern in $\{201,210\}$.

Conversely, suppose $e$ contains an occurrence $e_ke_le_m$ of one of the patterns in $\{201,210\}$. In  the $(k-1)$-th slice $U_{k-1}(\pi)$, there must exist three different intervals $I_a$, $I_b$ and $I_c$ in increasing order such that $I_a$ contains $\pi_k$. For otherwise, there has at most one interval upper than the internal containing $\pi_k$ and by the construction of $b(\pi)$, there will has at most one labeling smaller than $e_k$ that appears in the $a$-th slice $U_a(\pi)$ for any $a\geq k$, contradicting the fact that the suffix  $e_{k}e_{k+1}\cdots e_n$ contains a pattern in $\{201,210\}$. Now the occurrence of the intervals $I_a$, $I_b$ and $I_c$ implies that there exists
\begin{itemize}
\item a letter $\pi_i$ for some $i<k$ such that $\max(I_a)<\pi_i<\min(I_b)$;
\item a letter $\pi_j$ for some $j<k$ such that $\max(I_b)<\pi_i<\min(I_c)$;
\item a letter $\pi_{l'}$ for some $l'>k$ such $\pi_{l'}\in I_b$;
\item and a letter $\pi_{m'}$ for some $m'>k$ such that $\pi_{m'}\in I_c$.
\end{itemize}
The five letters $\pi_i$, $\pi_j$, $\pi_{k}$, $\pi_{l'}$ and $\pi_{m'}$ forms a pattern in $\{24135,24153,42135,42153\}$ such that $\pi_k$ plays the role of $1$.
\end{proof}

\begin{lemma}\label{inv:24135}
The set $\SS_n(24135,24153,42135,42153)$ is invariant under the MFS-action.
\end{lemma}
\begin{proof}
Suppose that $\pi\notin\SS_n(24135,24153,42135,42153)$, then $\pi$ contains subsequence $\pi_i\pi_j\pi_k\pi_l\pi_m$ that is order isomorphism to one pattern in  $\{24135,24153,42135,42153\}$. By the definition of $\varphi_a'(\pi)$, the letters $\pi_i$ and $\pi_j$ (resp.~$\pi_k$ and $\pi_l$) in $\varphi_a'(\pi)$ are still appear to the left (resp.~right) of $\pi_k$. Thus, the four letters $\pi_i$, $\pi_j$, $\pi_k$, $\pi_l$ and $\pi_m$ in $\varphi_a'(\pi)$ still form a pattern in $\{24135,24153,42135,42153\}$ and so $\varphi_a'(\pi)\notin\SS_n(24135,24153,42135,42153)$. The result then follows from the fact that $\varphi_a'$ is an involution on $\SS_n$ for any $a\in[n]$.
\end{proof}

\begin{proof}[{\bf Proof of Theorem~\ref{thm:gamcl}}]
By Lemmas ~\ref{act:gam} and  ~\ref{inv:24135}, we have
$$
\sum_{\pi\in\SS_n(24135,24153,42135,42153)}t^{\des(\pi)}=\sum_{k=0}^{\lfloor(n-1)/2\rfloor}|\widetilde{\SS_{n,k}}(24135,24153,42135,42153)|t^k(1+t)^{n-1-2k}.
$$
The $\gamma$-positivity expansion~\eqref{gam:cl} then follows by applying  identity~\eqref{24135}.
\end{proof}

\section*{Acknowledgement}
The first author  was supported by
the National Science Foundation of China grant 11701420.
The second author
was supported
by the National Science Foundation of China grant 11871247 and the project of Qilu Young Scholars of Shandong University.

\end{document}